\newif\ifsiam

\siamfalse

\ifsiam
   \documentclass[review]{siamart190516}
   \usepackage[numbers]{natbib}
\else
   \documentclass{amsart}
   \usepackage[margin=1in]{geometry}
   \usepackage{amsaddr,amsthm}
\fi

\usepackage{amsfonts,amsmath,amssymb}
\usepackage{bm}
\usepackage{booktabs}
\usepackage{caption}
\usepackage{hyperref}
\usepackage{cleveref} 
\usepackage{stmaryrd}
\usepackage{graphicx}
\usepackage{mathtools}
\usepackage{tikz}
\usepackage{tikz-cd}

\newtheorem{thm}{Theorem}
\newtheorem{cor}{Corollary}
\newtheorem{lem}{Lemma}
\newtheorem{rem}{Remark}
\newtheorem{prop}{Proposition}
\crefname{prop}{proposition}{propositions}

\def\ME{\widehat{M}}

\def\NC{\mathit{NC}}

\def\siamqedhere{\ifsiam\else\qedhere\fi}

\ifsiam
\title{Conservative and accurate solution transfer between high-order and low-order refined finite element spaces}
\author{
   Tzanio Kolev\thanks{Center for Applied Scientific Computing,
   Lawrence Livermore National Laboratory
   (\email{kolev1@llnl.gov}, \email{pazner1@llnl.gov}).}
   \and
   Will Pazner\footnotemark[1] \thanks{Corresponding author.}
}
\else
\title[HO--LOR Solution Transfer]{Conservative and accurate solution transfer between high-order and low-order refined finite element spaces}
\author{Tzanio Kolev and Will Pazner}
\address{Center for Applied Scientific Computing, Lawrence Livermore National Laboratory}
\fi

\begin{document}

\ifsiam\maketitle\fi

\begin{abstract}
In this paper we introduce general transfer operators between high-order and low-order refined finite element spaces that can be used to couple high-order and low-order simulations.
Under natural restrictions on the low-order refined space we prove that both the high-to-low-order and low-to-high-order linear mappings are conservative, constant preserving and high-order accurate.
While the proofs apply to affine geometries, numerical experiments indicate that the results hold for more general curved and mixed meshes.
These operators also have applications in the context of coarsening solution fields defined on meshes with nonconforming refinement.
The transfer operators for $H^1$ finite element spaces require a globally coupled solve, for which robust and efficient preconditioners are developed.
We present several numerical results confirming our analysis and demonstrate the utility of the new mappings in the context of adaptive mesh refinement and conservative multi-discretization coupling.
\end{abstract}

\ifsiam\else\maketitle\fi

\section{Introduction}

High-order numerical methods, including conforming finite elements, spectral elements, and discontinuous Galerkin methods, promise high efficiency and accurate solutions, in particular on modern computing architectures \cite{ceed_bp_paper_2020,Kronbichler2021,Hutchinson2016}.
However, traditional low-order methods remain useful for a large range of practical applications.
Furthermore, the development of stable high-order methods poses additional challenges \cite{CHQZ2,Anderson2014}, and their efficient implementation requires significant infrastructure \cite{Kronbichler2021, Anderson2020}.
For these reasons, while increasing number of components in large-scale simulation codes are transitioning to high-order, many other components remain low-order \cite{osiris,ceed-ms23,marbl}.
This is a considerable challenge for multi-physics simulations that need both types of components, since solution data must be exchanged between the high-order and low-order parts of the simulation.
Coupling high-order discretizations with low-order methods is also important for the purposes of preconditioning \cite{Orszag1980,Canuto2010,Pazner2020a,Bello-Maldonado2019}, shock capturing \cite{Persson2019,Ramirez2020}, and limiting \cite{Vilar2019,Pazner2020}, among others.
In the context of multigrid methods (in particular $p$-multigrid methods), prolongation and restriction operators are required to transfer solutions and residuals between high-order and low-order levels in the multigrid hierarchy \cite{Helenbrook2003,Helenbrook2006,Fidkowski2005,Sundar2015}.

The goal of this paper is to provide practical tools for solution transfer between high-order and low-order finite element spaces with supporting analysis that ensures both accuracy and conservation.
We propose transfer operators that are designed to preserve constant fields (cf.\ freestream preservation, \cite{Thomas1979,Kopriva2006}), conserve integrated quantities of interest (e.g.\ total mass, momentum, and energy), and retain as high of approximation properties as allowed by the given spaces.
We develop a general methodology to define such operators in the abstract setting, and then consider the specific case where the low-order space is obtained by the so-called \textit{low-order refined} procedure, whereby each element of the high-order mesh is subdivided into a number of subelements of lower order.

The high-to-low-order transfer operator $R$ and the low-to-high-order transfer operator $P$ are defined by simple variational problems \eqref{eq:R}--\eqref{eq:P}, and explicit formulas \eqref{eq:Rmat}--\eqref{eq:Pmat}, that can be implemented easily and efficiently in high-order application codes.
Assuming only that the low-order space is large enough (in the sense of condition \eqref{eq:VL}), we prove that the solution transfer will be conservative and constant preserving in both directions.
For the important practical case of a low-order refined space on a mesh with tensor product elements, we prove both conservation and high-order accuracy provided sufficient refinement (cf.\ \Cref{lem:orthogonality}), with spacing based on a Gauss--Lobatto-like quadrature rule.
Both of these requirements are natural and easy to satisfy.
Our approach works in 2D and 3D, on tensor-product and mixed meshes, and can be extend to high-order curved meshes, in which case one needs to choose between conservation and constant preservation.
While the proofs of the theoretical properties hold for meshes consisting of affine tensor-product elements, our numerical results indicate that these properties do generalize to the high-order curved and mixed meshes.
Besides coupling high-order and low-order codes, the transfer operators can also be used in other applications, e.g.\ coarsening in adaptive mesh refinement, an example of which is illustrated in the numerical results.

The rest of the paper is organized as follows.
In \Cref{sec:mappings}, we define the transfer operators in the abstract setting, and prove several important properties, particularly regarding conservation.
The accuracy of the resulting operators, relying on analysis of one-dimensional quadrature rules, is considered in \Cref{sec:accuracy}.
Implementation details and several numerical experiments verifying the theoretical properties of the operators are provided in \Cref{sec:numerical}.
This section includes also a multi-discretization example illustrating the high-order coupling of a structured grid finite volume advection solver with a high-order finite element space.
We end with conclusions in \Cref{sec:conclusions}.

\section{Mappings between high-order and low-order refined spaces}
\label{sec:mappings}

In this section we introduce the mappings between high-order (HO) and low-order refined (LOR) spaces that we propose as general transfer operators for coupling high-order and low-order simulations.
Here we focus on the conservation and constant preservation properties of the mappings, their accuracy is discussed in the following \Cref{sec:accuracy}.

\subsection{General mappings}

We begin by defining the transfer operators in general abstract settings.
Let $V$ be a Hilbert space with inner product $(\cdot\,,\cdot)$, and let $V_H$ and $V_L$ denote finite-dimensional subspaces of $V$.
In this abstract setting, $V_H$ represents a ``high-order'' subspace of $V$, and $V_L$ represents a ``low-order refined'' subspace of $V$, with the only requirement being that $V_L$ is  sufficiently large, such that
\begin{equation} \label{eq:VL}
  V_H \cap V_L^\perp = \{ 0 \}.
\end{equation}
Our goal is to define transfer operators between $V_H$ and $V_L$ that are conservative, accurate, and constant preserving.
We define these operators, $R : V_H \to V_L$ and $P : V_L \to V_H$ as follows
\begin{equation} \label{eq:R}
  (R u_H, v_L) = (u_H, v_L) \qquad\text{for all $v_L \in V_L$},
\end{equation}
and
\begin{equation} \label{eq:P}
  (P v_L, R u_H) = (v_L, R u_H) \qquad\text{for all $u_H \in V_H$.}
\end{equation}
These operators are illustrated by the following diagram:
\[
\begin{tikzcd}[column sep=1cm,row sep=1cm]
   V_H \drar[shift left=0.85]{R} & \lar[swap]{P} V_L \dar[shift left=2]{Q=RP} \\
    & S \arrow[hook]{u}{} \ular[shift left=0.85]{R^{-1}}
\end{tikzcd}
\]
In the remainder of the paper, we refer to the operator $R$ as the \textit{restriction} operator, and the operator $P$ as the \textit{prolongation operator}.

\begin{thm} \label{thm:abstract-properties}
   Assume that $V_H$ and $V_L$ satisfy \eqref{eq:VL}, and introduce the subspace $S = R(V_H) \subseteq V_L$.
   Then the transfer operators $R$ and $P$ defined by \eqref{eq:R} and \eqref{eq:P} have the following properties:
   \begin{enumerate}
      \item $R$ is injective, and $R: V_H \mapsto S$ is a bijection.
      \item $P$ is surjective, and $P$ is a left inverse of $R$, i.e.\ $PR : V_H \to V_H$ is the identity operator.
      \item $Q = RP : V_L \mapsto S$ is a projection, and $P = R|_S^{-1} Q$.
      \item For any functions $1_L \in V_L$ and $1_S \in S$, we have the following conservation properties:
      \begin{equation} \label{eq:conservation}
         (R u_H, 1_L) = (u_H, 1_L)
         \qquad\text{and}\qquad
         (P v_L, 1_S) = (v_L, 1_S).
      \end{equation}
      \item For any $u \in V_H \cap V_L$, $u = R u = P u$.
   \end{enumerate}
\end{thm}
\begin{proof} \leavevmode
\begin{enumerate}
   \item
      First, $R$ is well-defined as the orthogonal projection between two finite dimensional subspaces of $V$ (see \eqref{eq:Rmat} for an explicit formula for its matrix representation).
      Suppose $R u_H = R w_H$ for some $u_H, w_H \in V_H$.
      Then, for any $v_L \in V_L$
      \[
         0 = (R u_H - R w_H, v_L) = (R u_H, v_L) - (R w_H, v_L) = (u_H - w_H, v_L)
      \]
      by \eqref{eq:R}.
      Since $V_H \cap V_L^\perp = \{ 0 \}$, we have $u_H = w_H$, proving injectivity.
      $S$ is defined as the range of $R$, and so $R : V_H \mapsto S$ is a bijection, completing the proof
      of the first property.
   \item
      Since $S$ is a finite dimensional subspace of $V_L$, given $v_L \in V_L$ there is an unique orthogonal projection onto $S$, i.e. there is $w_H \in V_H$ such that $(R w_H, R u_H) = (v_L, R u_H)$ for all $u_H \in V_H$. By \eqref{eq:R} this implies that $P v_L := w_H$ is well-defined (see \eqref{eq:Pmat} for an explicit formula for its matrix representation).
      Now, let $u_H \in V_H$ be given. For any $w_H \in V_H$ we have
      \[
         (R P R u_H, R w_H)
            = (P R u_H, R w_H)
            = (R u_H, R w_H)  \qquad\text{by \eqref{eq:R} and \eqref{eq:P}}
      \]
      which gives us $RPR u_H = R u_H$.
      By injectivity of $R$, this implies $PR u_H = u_H$ for any $u_H \in V_H$, and so $P$ is surjective, and $PR = I$, proving the second property.
   \item
      Note that, for $Q = RP$, $Q^2 = RPRP = RP = Q$, and so $Q$ is a projection, and $P=R|_S^{-1}Q$ by the definition of $Q$ and the bijectivity of $R: V_H \mapsto S$. This means that we can think of $P$ as a two-step process: an orthogonal projection from $V_L$ to $S$, followed by inversion with $R$.
   \item
      The conservation properties hold for any $1_L \in V_L$ and $1_S \in S$ from the definitions \eqref{eq:R} and \eqref{eq:P}.
      Note that $R$ has stronger conservation properties than $P$ because, in general, $V_L$ is larger than $S$.
   \item
      Finally, setting $u$ for $Ru$ clearly satisfies \eqref{eq:R}, and so $Ru = u$. Then $PR = I$ implies $Pu = PR u = u$, completing the proof. \siamqedhere
\end{enumerate}
\end{proof}

\Cref{thm:abstract-properties} shows that the restriction and prolongation operators defined by \eqref{eq:R} and \eqref{eq:P} have many desirable properties for coupling simulations posed in the $V_H$ and $V_L$ spaces.
For example, since $PR = I$, any ``high-order'' function mapped with $R$ can be exactly recovered by $P$ and so no information is lost by using the ``low-order'' space.
In addition, if constant functions belong to both $V_H$ and $V_L$, then both operators preserve them.
Furthermore, if $V=L^2(\Omega)$, both operators are conservative in the sense that they preserve the integrals over $\Omega$, e.g. the mass is preserved when transferring density.
Note than even if the condition \eqref{eq:VL} does not hold, the restriction operator $R$ still enables a one-way conservative map from high to low order.
However, when \eqref{eq:VL} is satisfied, we have a much more useful two-way coupling.

\begin{rem}[Transfer of dual vectors]
   In this paper, we focus on the transfer of \textit{primal vectors} between the spaces $V_H$ and $V_L$.
   However, we note that the operators $R$ and $P$ defined above can also be used to define transfer operators $R^* : V_H^* \to V_L^*$ and $P^* : V_L^* \to V_H^*$ between the corresponding dual spaces.
   Letting $M_H$ and $M_L$ denote the Riesz identification of a primal vector with its associated dual vector, we briefly discuss two possible definitions for these transfer operators.
   The first is given by
   \[
      R^*_1 = M_L R M_H^{-1}, \qquad\qquad P^*_1 = M_H P M_L^{-1},
   \]
   while the second is given by
   \[
      R^*_2 = P^T, \qquad\qquad P^*_2 = R^T.
   \]
   It is easy to see that in both cases $P^* R^* = I$, and both sets of operators satisfy constant preservation and conservation properties.
   The operators are illustrated by the following diagrams:
   \[
   \begin{tikzcd}[column sep=1cm,row sep=1cm]
      V_H \rar[shift left=0.85]{R} \dar[swap]{M_H} & \lar[shift left=0.85]{P} \dar{M_L} V_L \\
      V_H^* \rar[shift left=0.85]{R^*_1} & \lar[shift left=0.85]{P^*_1} V_L^*
   \end{tikzcd}
   \qquad
   \begin{tikzcd}[column sep=1cm,row sep=1cm]
      V_H \rar[shift left=0.85]{R} \dar[swap]{R^T M_L R} & \lar[shift left=0.85]{P} \dar{M_L} V_L \\
      V_H^* \rar[shift left=0.85]{R^*_2} & \lar[shift left=0.85]{P^*_2} V_L^*
   \end{tikzcd}
   \]
\end{rem}

\subsection{High-order and low-order refined mappings}

As a canonical example of the transfer operators $R$ and $P$, we consider the so-called ``low-order refined transfer.''
Let $\Omega \subseteq \mathbb{R}^d$, $d=1,2,\text{ or }3$ denote a spatial domain, and let $\mathcal{T}_H$ denote the computational mesh.
Let $V = L^2(\Omega)$, where $(u,\,v) = \int_\Omega u v \, dx$ denotes the standard $L^2$ inner product.
Let $V_H$ be a high-order finite element space, whose elements are piecewise polynomials of degree $p$.
The \textit{high-order space} $V_H$ can be either a continuous or discontinuous space (e.g.\ $V_H$ is either an $H^1$ or $L^2$ finite element space).
The \textit{low-order refined} space $V_L$ is a finite element space with polynomial degree $q \leq p$, defined on a mesh $\mathcal{T}_L$, obtained by refining the high-order mesh, $\mathcal{T}_H$, $n$ times in each dimension.
An illustration of these spaces is given in \Cref{fig:ho-lor-mesh}.
We remark that such low-order refined spaces has been used extensively in the context of preconditioning (with $R=P=I$), where the spectral equivalence of the mass and stiffness matrices defined on $V_L$ and $V_H$ is often referred to as the finite element method--spectral element method (FEM--SEM) equivalence \cite{Casarin1997,Canuto1994,Canuto2010}.

\begin{figure}
   \newlength{\fw}
   \ifsiam\setlength{\fw}{0.3\linewidth}\else\setlength{\fw}{0.3\linewidth}\fi
   \hspace{0.2\fw}
   \begin{minipage}{\fw}
      \centering
      \includegraphics[width=\linewidth]{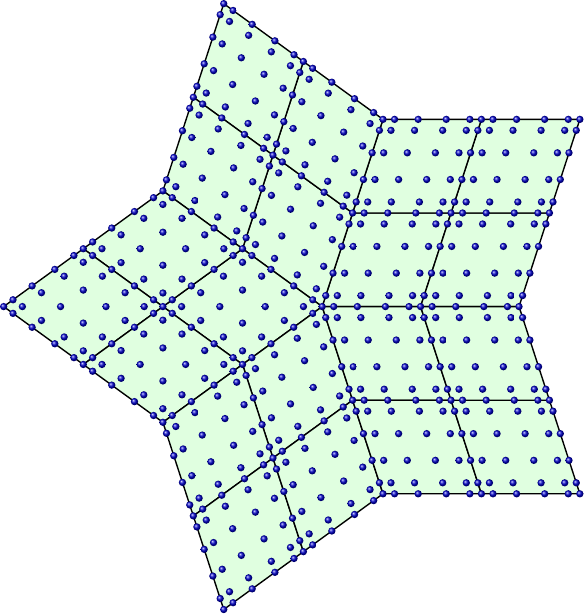}
      \[ V_H \]
   \end{minipage}
   \hspace{1in}
   \begin{minipage}{\fw}
      \centering
      \includegraphics[width=\linewidth]{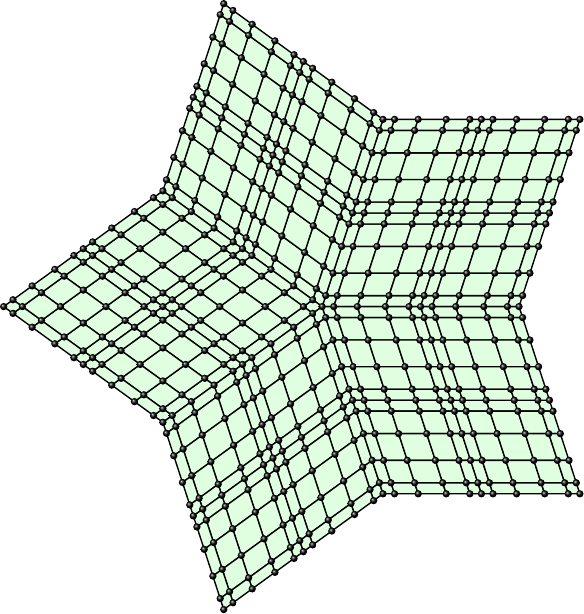}
      \[ V_L \]
   \end{minipage}
   \hspace{0.2\fw}

   \caption{
      Illustration of high-order ($V_H$) and low-order refined ($V_L$) finite element spaces.
      The high-order space is defined on a course mesh with polynomial degree $p=5$, indicated by the blue nodal points (left panel).
      The low-order space is defined on a mesh obtained by subdividing each coarse element into subelements with Gauss--Lobatto points as vertices (right panel).
   }
   \label{fig:ho-lor-mesh}
\end{figure}

We require that the mesh $\mathcal{T}_L$ be sufficiently refined so that the number of degrees of freedom in element of $V_H$ is less than the total number of $V_L$ degrees of freedom in that element.
While $V_L$ can be chosen to be either a continuous or discontinuous space, it is typically more computationally efficient if the low-order space is discontinuous, and we will make this assumption in many of the examples below.
We next show that under these assumptions (with $V_L$ discontinuous) the orthogonality assumption \eqref{eq:VL} is satisfied, namely $V_H \cap V_L^\perp = \{ 0 \}$, cf.\ \cite{Chihara1978},
and thus the statements of \Cref{thm:abstract-properties} hold.

\begin{lem} \label{lem:orthogonality}
   Let $u(x)$ be a polynomial of degree $p$ defined on $[-1,1]$.
   Decompose the interval $[-1,1]$ into $n$ subintervals $[a_i,a_{i+1}]$.
   Furthermore, suppose that for all such subintervals, we have
   \[
      \int_{a_i}^{a_{i+1}} u(x) Q(x) \, dx = 0 \qquad \deg(Q) \leq q,
   \]
   where
   \begin{equation} \label{eq:nqp}
     n (q+1) \geq p+1
   \end{equation}
   Then, $u(x)$ is identically zero.
\end{lem}
\begin{proof}
   Fix one subinterval $[a_i, a_{i+1}]$.
   Let $x_1, x_2, \ldots, x_p$ denote the zeros of $u$, ordered such that the first $m$ zeros are those of odd multiplicity contained in $(a_i, a_{i+1})$, and the remaining $p+1-m$ zeros are either of even multiplicity, or lie outside of $(a_i, a_{i+1})$.
   We claim that $m \geq q + 1$.
   Suppose to the contrary that $m \leq q$, and define the polynomial $Q(x)$ by
   \[
      Q(x) = \prod_{i=1}^{m} (x - x_i).
   \]
   Then, by assumption,
   \[
      \int_{a_i}^{a_{i+1}} u(x) Q(x) \, dx = 0.
   \]
   Note that $u(x)Q(x)$ has only zeros of even multiplicity, and so it does not change sign on $(a_i, a_{i+1})$.
   Therefore, either $u \equiv 0$ or else we obtain a contradiction, and conclude that $m \geq q + 1$.
   In the latter case, we see that $u$ has at least $q+1$ zeros in each interval, and so $u$ has at least $n (q+1) \geq p+1$ zeros in $[-1,1]$, implying that $u \equiv 0$.
\end{proof}

\begin{rem} \label{rem:mapping}
   The above theorem can be generalized to the $d$-dimensional cube $[-1,1]^d$ using a tensor-product argument.
   The same conclusion holds if the integral is weighted with a separable tensor-product weight.
   In particular, if $V_H$ is a degree-$p$ finite element spaces defined on a mesh consisting of affine tensor-product elements, and $V_L$ is a degree-$q$ finite element space defined on a mesh refined $n$ times in each dimension, then $V_H \cap V_L^\perp = \{ 0 \}$ if $n (q+1) \geq p+1$.
   In \Cref{sec:numerical}, we numerically study the generalization to non-affine meshes, curved meshes, and simplex elements.
\end{rem}

Since any constant function belongs to both $V_L$ and $V_H$, the transfer operators $R$ and $P$ preserve constants, and by taking $1_L = 1_S = 1$, the constant functions, we have the following conservation properties from \eqref{eq:conservation}:
\[
   \int_\Omega R(u_H) \, dx = \int_\Omega u_H \, dx
   \qquad\text{and}\qquad
   \int_\Omega P(v_L) \, dx = \int_\Omega v_L \, dx.
\]
When $V_H$ is discontinuous, $R$ and $P$ preserve piecewise-constants on $\mathcal{T}_H$, and we get the stronger local conservation properties
\[
   \int_{e_H} R(u_H) \, dx = \int_{e_H} u_H \, dx
   \qquad\text{and}\qquad
   \int_{e_H} P(v_L) \, dx = \int_{e_H} v_L \, dx.
\]
for any element $e_H \in \mathcal{T}_H$.

\begin{rem}[Curved meshes]
   In many applications, the high-order space $V_H$ is defined on a mesh with curved elements.
   If the low-order refined space $V_L$ is also defined on a mesh with the same curved elements, then the above analysis holds.
   However, it is also practical for the low-order refined space $V_L$ to be defined on a straight-sided low-order mesh.
   In this case, the areas and volumes of the mesh elements are not the same, and so one can have conservation or constant preservation but not both \cite{Anderson2014}.
   Our default option is to choose the former, which results in a second-order error in the mass conservation, see \Cref{sec:curved} for numerical results.
\end{rem}

\begin{rem}[AMR derefinement]
Suppose that $V_L$ is obtained from the finite element space $V_H$ through an adaptive mesh refinement procedure.
This procedure can include non-conforming refinement (i.e.\ with hanging nodes \cite{Cerveny2019}).
Note that $V_H \subseteq V_L$, and so $R$ is given by the natural injection.
In this case, the $P$ operator can be used to \textit{derefine} functions defined on the adaptively refined mesh, see \Cref{sec:amr} for numerical results.
\end{rem}

In the case of high-order and low-order refined finite element spaces, the transfer operators $R$ and $P$ can be expressed naturally in terms of the low-order and mixed mass matrices.
Let $M_L$ denote the low-order mass matrix, i.e.,
$$
(M_L)_{ij} = \int_\Omega \psi^L_i \psi^L_j,
$$
where $\{\psi^L_j\}$ are the basis functions for $V_L$, and let $M_{LH}$ denote the mixed mass matrix, i.e.
$$
(M_{LH})_{ik} = \int_\Omega \psi^L_i \psi^H_k,
$$
where $\{\psi^H_k\}$ are the basis function for $V_H$.
Then, by definition \eqref{eq:R}, the operator $R : V_H \to V_L$ can be written in matrix form
(operating on the vectors of degrees of freedom in $V_H$ and $V_L$) as
\begin{equation} \label{eq:Rmat}
   R = M_L^{-1} M_{LH}.
\end{equation}
If $V_L$ is a discontinuous space, then $M_L$ is block-diagonal, and so $M_L^{-1}$ can be computed efficiently element-by-element.
In the piecewise-constant LOR case, $M_L$ is a diagonal matrix, and can be inverted trivially.

Since \eqref{eq:VL} implies that $R$ is injective, the associated matrix has full column rank, and so $R^T M_L R$ is invertible.
Then, the operator $P : V_L \to V_H$ defined by \eqref{eq:P} has the matrix representation
\begin{equation} \label{eq:Pmat}
   P = (R^T M_L R)^{-1} R^T M_L.
\end{equation}
Note that this explicit matrix representation immediately gives $PR = I$ (cf.\ property 2 of \Cref{thm:abstract-properties}).
If the high-order space is also discontinuous, then $P$ can be computed efficiently element-by-element.
Otherwise, a global solve is required to compute the action of $P$.

\begin{rem}[Alternative transfer operators]
   Given a restriction operator $R : V_L \to V_H$, it is possible to define several alternative prolongation operators $P : V_L \to V_H$.
   In this work, we choose $P$ to be given by \eqref{eq:P} since it is a conservative, constant-preserving left-inverse of $R$.
   One other natural choice of prolongation operator is the $L^2$ projection, $P' = M_H^{-1} M_{HL} = M_H^{-1} R^T M_L$, which is conservative and constant-preserving; however, $P'$ fails to be a left-inverse for $R$.
   Similarly, the prolongation operator given by $P'' = (R^T R)^{-1} R^T$ is a left-inverse for $R$, but in general fails to be conservative.
   Pointwise nodal interpolation is also commonly used, particularly in the context of FEM--SEM preconditioning, but this operation is not conservative and depends on the choice of nodal interpolation points for the high-order space.

   Another option similar to what we propose in this paper is to define $\hat{P} : V_L \to V_H$ as the $L^2$ projection, and then define $\hat{R} : V_H \to V_L$ to be its conservative right-inverse $\hat{R} = R (R^T M_L R)^{-1} M_H$.
   This definition gives rise to operators with properties similar to those enumerated in \Cref{thm:abstract-properties}.
\end{rem}

\subsubsection{Preconditioning the \texorpdfstring{$P$}{P} operator}
\label{sec:preconditioning-P}

In the case where the high-order space $V_H$ is continuous, a globally coupled solve is required to compute the action of $P = (R^T M_L R)^{-1} R^T M_L$.
The following result establishes that $R^T M_L R$ is spectrally equivalent to the high-order mass matrix,
$$
(M_H)_{lk} = \int_\Omega \psi^H_l \psi^H_k,
$$
and so any effective preconditioner for $M_H$ can be used to precondition the inversion of the $R^T M_L R$ operator in the action of $P$.

\begin{prop} \label{prop:preconditioning-P}
   Let $V_L$ be a low-order refined finite element space, let $M_L$ and $M_H$ denote the low-order and high-order mass matrices, respectively, and let $A = R^T M_L R$, where $R$ is defined by \eqref{eq:R}.
   Then, $M_H^{-1} A$ is uniformly well-conditioned.
\end{prop}
\begin{proof}
   Note that the operator $R$ is given by the restriction to $V_H$ of the $L^2$ projection onto $V_L$.
   Therefore, $R$ is a projection, and so, for any $v_H \in V_H$, $\| R v_H \|_0 \leq \| v_H \|_0$. Hence
   \[
      v_H^T R^T M_L R v_H = \| R v_H \|_0^2 \leq \| v_H \|_0^2.
   \]
   Since $R$ is injective, we have $\| R v_H \| \geq \alpha \| v_H \|_0$, where the lower bound $\alpha$ is estimated in \Cref{lem:R-lower-bound}, and shown to be independent of polynomial degree in the context of piecewise constant low-order refined space $V_L$.
   Therefore, $R^T M_L R \sim M_H$, and so the condition number of $M_H^{-1} R^T M_L R$ is uniformly bounded.
\end{proof}

In the special case of tensor-product meshes (i.e.\ with quadrilateral or hexahedral elements), the high-order mass matrix is spectrally equivalent to its diagonal, independent of mesh size $h$ and polynomial degree $p$, see \cite{Canuto1994,Teukolsky2015} and \Cref{sec:ho-lor-equivalence}.
As a consequence, the above proposition implies that the operator $A = R^T M_L R$ is well-preconditioned by the diagonal $D$ of the high-order mass matrix in this case, enabling efficient and readily available diagonal preconditioning.

\begin{cor} \label{cor:preconditioning-P}
   Consider tensor-product finite element spaces $V_L$ and $V_H$ with Gauss--Lobatto nodal basis functions.
   Let $D$ denote the diagonal of the high-order mass matrix $M_H$, and let $A = R^T M_L R$ as in \Cref{prop:preconditioning-P}.
   Then, $D^{-1} A$ is uniformly well-conditioned.
\end{cor}

\subsubsection{Conservation of multiple fields}

In certain contexts, it may be desirable to conservatively transfer multiple fields.
For example, suppose that density and velocity are represented in the high-order spaces as $\rho_H$ and $u_H$, respectively.
We wish to compute low-order approximations, $\rho_L$ and $u_L$ that are both mass and momentum conserving.
In many applications, the density $\rho_H \in V_{\rho,H}$ is discontinuous, and the velocity $u_H \in V_{u,H}$ is continuous, see e.g. \cite{Anderson2018}.
Let $R_\rho : V_{\rho,H} \to V_{\rho,L}$ denote the restriction operator defined above in terms of the standard $L^2$ inner product, and let $\rho_L = R(\rho_H)$.
Then, mass conservation follows from \eqref{eq:R}.

In order to define the momentum-conserving transfer operator for velocity, we consider density-weighted inner products on the velocity spaces:
\[
   (u_L, v_L)_u = \int_\Omega u_L v_L \rho_L \, dx
   \qquad\text{and}\qquad
   (u_H, v_L)_u = \int_\Omega u_L v_L \rho_H \, dx.
\]
Having first computed $\rho_L = R(\rho_H)$, we can compute the density-weighted transfer operator $R_\rho$ using these weighted inner products.
Then, letting $u_L = R_\rho(u_H)$, we have, by \eqref{eq:R},
\[
   \int_\Omega u_L \rho_L \, dx
   = (u_L, 1)_u
   = (u_H, 1)_u
   = \int_\Omega u_H \rho_H \, dx,
\]
proving conservation of momentum.
If the low-order space is discontinuous, then the computation of $R_\rho$ requires only the inversion of the block-diagonal density-weighted mass matrix.

A similar procedure can be used to map from the low-order spaces to the high-order spaces.
First, given the low-order density $\rho_L$, the high-order density $\rho_H = P(\rho_L)$ is computed using the prolongation operator $P$ defined in terms of the unweighted $L^2$ inner product.
Assuming that the density spaces $V_{\rho,L}$ are $V_{\rho,H}$ are discontinuous, we see that $P = (R^T M_L R)^{-1} R^T M_L$ can be computed element-by-element.
Once $\rho_H = P(\rho_L)$ is computed, we can compute the density-weighted prolongation operator $P_\rho$ in terms of the density-weighted inner products.
Since the high-order velocity space $V_{u,H}$ is typically continuous, the matrix $R_{\rho}^T M_L R_{\rho}$ is not block-diagonal, and so the corresponding prolongation operator cannot be computed element-by-element, and instead $R_{\rho}^T M_L R_{\rho}$ may be preconditioned using the results of \Cref{prop:preconditioning-P} and \Cref{cor:preconditioning-P}.

\section{Accuracy of the mappings}
\label{sec:accuracy}

In this section we study the accuracy of the transfer operators \eqref{eq:R} and \eqref{eq:P}.
$R$ is a standard $L^2$ projection operator with respect to the inner product on $V_L$, and its accuracy properties are well understood.
Therefore, our focus is on the accuracy of the prolongation operator $P$.

Let $f \in V$ be a given function, and suppose that $f$ is approximated by $f_H \in V_H$, i.e.\ $f = f_H + e_H$, for some error term $e_H$.
Let $f_L = \Pi_L f$, where $\Pi_L$ denotes $L^2$ projection onto $V_L$.
We are interested in the accuracy of $P f_L \in V_H$ compared with $f_H$.
We begin with a general result, estimating the accuracy of $P f_L$ in terms of a lower bound for $R$.

\begin{lem}
   Let $f \in V$ be given, and let $f_H \in V_H$.
   Define $f_L \in V_L$ by $f_L = \Pi_L f$.
   Then,
   \[
      \| P f_L - f \| \leq (1 + \alpha^{-1}) \| f - f_H \|,
   \]
   where $\alpha$ gives a lower bound for the operator $R$, i.e.\ $\| R v \| \geq \alpha \| v \|$.
\end{lem}
\begin{proof}
   First, note that $\Pi_L f = \Pi_L f_H + \Pi_L e_H$, and so $P f_L = P \Pi_L f_H + P \Pi_L e_H$.
   Recall from the definition \eqref{eq:P}, that $P \Pi_L f_H$ is defined by
   \[
      (P \Pi_L f_H, R u_H) = (\Pi_L f_H, R u_H) \qquad \text{for all $u_H\in u_H$.}
   \]
   By definition of the $L^2$ projection,
   $
      (\Pi_L f_H, R u_H) = (f_H, R u_H),
   $
   and so by \eqref{eq:R}
   \[
      (R P \Pi_L f_H, R u_H) = (R f_H, R u_H),
   \]
   implying (since $R$ is injective) that $P \Pi_L f_H = f_H$.
   Therefore
   \begin{equation} \label{eq:err}
      \| P f_L - f \| = \| f_H + P \Pi_L e_H  - f \| \leq \| P \Pi_L e_H \| + \| e_H \|,
   \end{equation}
   and so it remains to estimate the term $\| P \Pi_L e_H \|$.
   Since $\Pi_L$ is a projection, we have $\| \Pi_L e_H \| \leq \| e_H \|.$
   Furthermore $RP$ is a projection by \Cref{thm:abstract-properties}, and so  $\| R P u_L \| \leq \| u_L \|$.
   Since the operator $R$ is injective, we have $\| R u_H \| \geq \alpha \| u_H \|$ for some $\alpha$, and hence
   \[
      \| P u_L \| \leq \frac{1}{\alpha} \| u_L \|.
   \]
   Combining this estimate with \eqref{eq:err}, we have
   \[
      \| P f_L - f \| \leq \| e_H \| + \| P \Pi_L e_H \| \leq \| e_H \| + \frac{1}{\alpha} \|e_H\| = (1 + \alpha^{-1}) \| e_H \|. \siamqedhere
   \]
\end{proof}

Now, we consider the specific case of high-order to low-order refined transfer defined on meshes with tensor-product elements.
Suppose $V = L^2(\Omega)$, and $V_H$ is a finite element space with polynomial degree $p$ and mesh size $h$, and $V_L$ is a low-order refined piecewise constant finite element space.
Then, if $f$ possesses sufficient regularity, we can bound the error term $\| e_H \|$ by $\mathcal{O}(h^{p+1})$.
Additionally, in this case, we have the following lower bound on $R$, whose proof we defer to the following subsections.

\begin{lem} \label{lem:R-lower-bound}
  Let $V_H$ be a high-order finite element space of degree $p$ consisting of affine tensor-product elements, and let $V_L$ be a piecewise constant low-order refined finite element space satisfying $V_H \cap V_L^\perp = \{ 0 \}$.
  Then, the operator $R$ has the lower bound
   \[
      \| R v \| \geq \alpha \| v \| \qquad \text{for all $v \in V_H$,}
   \]
   where the constant $\alpha \sim 1$ is independent of the polynomial degree $p$.
\end{lem}

As a consequence of the above two lemmas, we have the following main accuracy result regarding the prolongation operator $P$.
Informally it states that the range of $P$ has the same approximation properties as the full high-order space, and thus there is
no loss of high-order accuracy from the transfers between the high-order and low-order refined spaces.

\begin{thm}
   \label{thm:P-accuracy}
   Let $V_H$ be a high-order finite element space consisting of affine tensor-product elements, and let $V_L$ be a piecewise constant low-order refined finite element space satisfying $V_H \cap V_L^\perp = \{ 0 \}$.
   Let $f \in V$ be given, and sufficiently regular, such that $f = f_H + e_H$, where $e_H = \mathcal{O}(h^{p+1})$.
   Let $f_L = \Pi_L f$.
   Then,
   \[
      \| P f_L - f \|_0 \lesssim h^{p+1} \| f \|_0,
   \]
   where the implied constant is independent of the polynomial degree of the high-order space.
\end{thm}

We now turn our attention to the proof of \Cref{lem:R-lower-bound}.
We begin by enumerating some technical results regarding one-dimensional quadrature rules.

\subsection{1D quadrature analysis}

In this section we derive some estimates for the abscissas $x_i$ and weights $w_i$ of 1D quadrature rules
\[
   \int_{-1}^{1} f(x) dx \approx \sum_i w_i f(x_i)
\]
which we classify to be of either {\em open} or {\em closed} type:
\begin{enumerate}
   \item The \textit{open} rules have $n$ abscissas and weights $\{ (w_i,x_i)\}_{i=1}^n$, with all points being interior to the interval.
   Examples include the Gauss and Chebyshev (Fejer's first) rules.
   \item The \textit{closed} rules have $n+1$ abscissas and weights $\{ (w_i,x_i) \}_{i=0}^n$, which include the points $x_0=-1$ and $x_n=1$.
   Examples include the Gauss--Lobatto and Chebyshev--Lobatto (Clenshaw--Curtis) rules.
\end{enumerate}

We assume that the abscissas are sorted in an increasing order, and that the rule is \textit{symmetric} with respect to the origin, i.e.\ $x_i = -x_{n-i+1}$ and $w_i = w_{n-i+1}$ in the open case, and similar in the closed case.
We will also use superscripts to distinguish between the different quadrature rules, e.g.\ $G$ for Gauss, $C$ for Chebyshev, $GL$ for Gauss--Lobatto, and $CL$ for Chebyshev--Lobatto.

\begin{rem} \label{rem:abscissas}
   Generally, the abscissas of the open rules are the zeros of orthogonal polynomials with certain weights, while the closed abscissas are the zeros of the derivative of that polynomial plus the two endpoints, $\pm 1$.
   Specifically
   \begin{enumerate}
      \item The Gauss points $\{ x_i^G \}_{i=1}^n$ are the zeros of the Legendre polynomials $P_n(x)$ which are orthogonal in $[-1,1]$ with weight $1$.
      \item The Chebyshev points $\{ x_i^C \}_{i=1}^n$ are the zeros of the Chebyshev polynomials $T_n(x)$ which are orthogonal in $[-1,1]$ with weight $\frac{1}{\sqrt{1-x^2}}$.
         We have $ x_i^C = -\cos \phi_i$, $\phi_i = (i-\frac{1}{2})\pi/n$.
      \item The Gauss--Lobatto points $\{ x_i^{GL} \}_{i=0}^n$ are the zeros of $(1-x^2)P'_n(x)$.
      \item The Chebyshev--Lobatto points $\{ x_i^{CL} \}_{i=0}^n$ are the zeros of $(1-x^2)T'_n(x)$.
         We have $ x_i^{CL} = -\cos \phi_i$, $\phi_i = i \pi / n$.
   \end{enumerate}
   Both Legendre and Chebyshev are special cases of the Jacobi polynomials $P_n^{(\alpha, \beta)}$ which are orthogonal on $[-1,1]$ with respect to the weight $(1-x)^\alpha(1+x)^\beta$.
   The properties of the Jacobi polynomials are critical for the estimates in this section, see \cite{Szego1939}.
\end{rem}

We use the notation $x \sim y$ to denote that there are constants $0 < c < C$ independent of the number of quadrature points, such that $c x \leq y \leq C x$.
We first observe that on $[0,\pi/2]$ the functions $x$ and $\sin(x)$ are equivalent with respect to the $\sim$ relation.

\begin{lem} \label{lem:sin}
   If $x$ and $y$ are in $[0,\pi/2]$, then
   $
      x \sim \sin x
   $
   and
   $
      x \sim y \implies \sin x \sim \sin y.
   $
\end{lem}

In the next proposition we summarize a number of known estimates of quadrature weights and points based on representation of the points via angles on the unit semi-circle.

\begin{prop} \label{prop:asymp}
   The $n$ Gauss weights and points satisfy
   \[
      x^G_i = -\cos \phi^G_i,\qquad
      w^G_i \sim \frac{\pi}{n} \sin \phi^G_i,\qquad
      \phi^G_{i} \sim \frac{2i-1}{2n} \pi \,,
      \qquad i=1, \ldots, n \,.
   \]
   Similarly, the $n+1$ Gauss--Lobatto weights and points satisfy
   \[
      x^{GL}_i = -\cos \phi^{GL}_i,\qquad
      w^{GL}_i \sim \frac{\pi}{n} \sin \phi^{GL}_i,\qquad
      \phi_i^{GL} \sim \frac{i}{n} \pi \,,
      \qquad i=0, \ldots, n.
   \]
   Additionally
   \[
      \phi^{G}_{i} - \phi^{G}_{i-1} \sim  \frac{\pi}{n}
      \qquad\text{and}\qquad
      \phi^{GL}_{i} - \phi^{GL}_{i-1} \sim  \frac{\pi}{n}
   \]
   for $ i=2, \ldots, n$ and $ i=1, \ldots, n$ respectively.

   The quantities $\phi_i^G$ and $\phi_i^{GL}$ are referred to as the \textbf{Gauss angles} and \textbf{Gauss--Lobatto angles}, respectively.
\end{prop}
\begin{proof}
   The estimate $w_i \sim \frac{\pi}{n} \sin \phi_i$ can be found in the form
   \[
      w_i \sim \frac{\sqrt{1- x_i^2}}{n}
   \]
   for both the Gauss and Gauss--Lobatto weights as equation (2.3.16) in \cite{CHQZ2}, and for the Gauss--Lobatto weights as equation (2.3) in \cite{Canuto1994}.
   Both Gauss--Lobatto estimates are in the case $i=1, \ldots, n-1$.
   For the special case of the endpoint weights we have
   \[
   w_0^{GL} = w_n^{GL} = \frac{2}{n(n+1)} \sim \frac{1}{n^2} \,.
   \]
   These estimates are derived from Darboux's asymptotic formulas for Jacobi polynomials, see (15.3.10) in \cite{Szego1939} for the Gauss case (in the setting of that paper $\alpha=\beta=0$, $\lambda_\nu=w_i$, and $\theta_\nu=\phi_i$).

   The Bruns estimates for the Gauss angles (cf.\ (6.6.2) in \cite{Szego1939} and \Cref{fig:bruns-sundermann}) are:
   \begin{equation} \label{eq:bruns}
      \frac{2i-1}{2n+1}\pi \leq \phi^G_i \leq \frac{2i}{2n+1} \pi \,,  \qquad i=1, \ldots, n,
   \end{equation}
   from which we obtain
   \[
      \frac{1}{2n+1}\pi \leq \phi^{G}_i - \phi^{G}_{i-1} \leq \frac{3}{2n+1} \pi \,,
   \]
   and therefore
   \[
      c \frac{\pi}{n} \leq \phi^{G}_i - \phi^{G}_{i-1}\leq C \frac{\pi}{n}
   \]
   for $c=\frac{1}{3}$ and $C=\frac{3}{2}$.
   Furthermore, \eqref{eq:bruns} implies
   \[
      c \frac{2i-1}{2n}\pi \leq \phi^G_i \leq C \frac{2i-1}{2n} \pi
   \]
   for $c=\frac{2}{3}$ and $C=3$.

   \begin{figure}
      \centering
      \ifsiam
      \includegraphics[width=\linewidth]{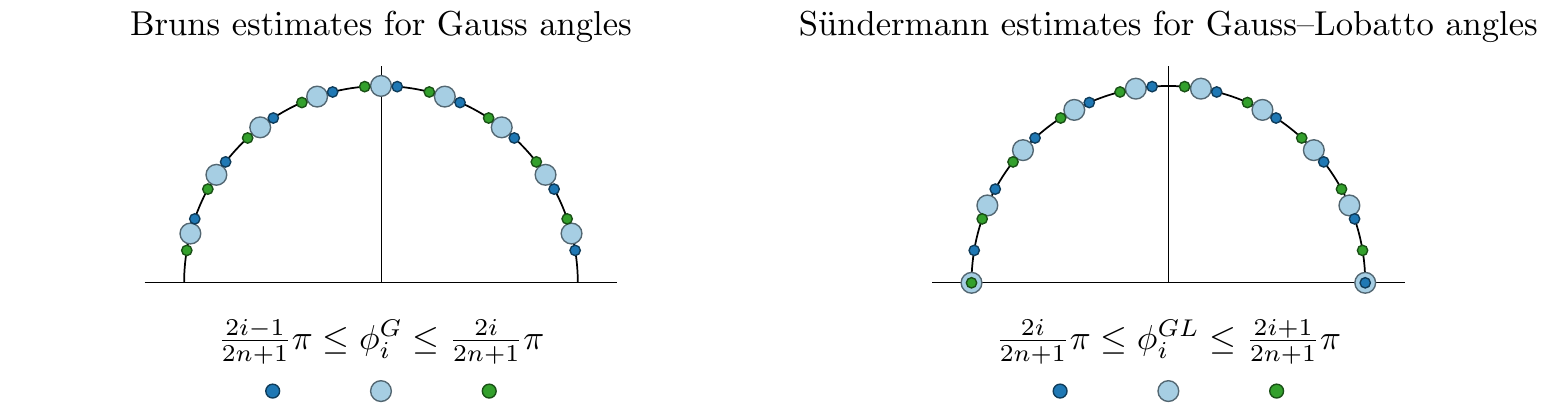}
      \else
      \includegraphics[width=\linewidth]{fig/point_estimates}
      \fi
      \caption{Bruns estimates for the Gauss angles (left) and S\"undermann estimates for the Gauss--Lobatto angles (right) in the case $n=9$.}
      \label{fig:bruns-sundermann}
   \end{figure}

   The S\"undermann estimates for the Gauss--Lobatto angles (cf.\ \cite{Sundermann1983}, \cite[Lemma 4.2]{Melenk02} and \Cref{fig:bruns-sundermann}) are
   \begin{equation} \label{eq:sundermann}
   \frac{2i}{2n+1}\pi \leq \phi_i^{GL} \leq \frac{2i+1}{2n+1} \pi \,,  \qquad i=0, \ldots, n
   \end{equation}
   from which we obtain
   \[
   \frac{1}{2n+1}\pi \leq \phi^{GL}_i - \phi^{GL}_{i-1} \leq \frac{3}{2n+1} \pi \,,
   \]
   and therefore
   \[
   c \frac{\pi}{n} \leq \phi^{GL}_i - \phi^{GL}_{i-1}\leq C \frac{\pi}{n}
   \]
   for $c=\frac{1}{3}$ and $C=\frac{3}{2}$.
   Since $\phi_0^{GL} = 0$, it is sufficient to show $\phi_i^{GL} \sim \frac{i}{n} \pi$ for $i \geq 1$.
   In that case, \eqref{eq:sundermann} implies
   \[
   c \frac{i}{n}\pi \leq \phi_i^{GL} \leq C \frac{i}{n}\pi \,,
   \]
   for $c=\frac{2}{3}$ and $C=\frac{3}{2}$.
\end{proof}

   The Chebyshev and Chebyshev--Lobatto points are defined as $x_i^C = -\cos \phi^C_i$, $\phi^C_i = \frac{2i-1}{2n} \pi$ and $x_i^{CL} = -\cos \phi^{CL}_i$, $\phi^{CL}_i = \frac{i}{n} \pi$, so the angle equivalences in \Cref{prop:asymp} hold as equalities.
   Equivalently, we have
   \[
      \phi^G_{i} \sim  \phi^C_i
      \qquad\text{and}\qquad
      \phi^{GL}_{i} \sim  \phi^{CL}_i \,.
   \]

   We also note the useful property that the $n$ Gauss points $\{x_i^G\}_{i=1}^n$ interleave the $n+1$ Gauss--Lobatto points $\{x_i^{GL}\}_{i=0}^n$, which is a simple consequence of the definitions in \Cref{rem:abscissas}.

\begin{prop} \label{prop:wh-estimates}
   For the points and weights in $[-1,0]$ we have
   \[
      w^G_i \sim (i-1/2) \frac{\pi^2}{n^2}
      \qquad\text{and}\qquad
      h_i^G := x^{G}_{i} - x^{G}_{i-1} \sim (i-1) \frac{\pi^2}{n^2}
   \]
   for $i=1, \ldots, \lceil n/2 \rceil$, and
   \[
      w^{GL}_i \sim i \frac{\pi^2}{n^2}
      \qquad\text{and}\qquad
      h_i^{GL} := x^{GL}_{i} - x^{GL}_{i-1} \sim (i-1/2) \frac{\pi^2}{n^2}
   \]
   for $i=1, \ldots, \lceil n/2 \rceil$ and $w^{GL}_0 \sim 1/n^2$.
   It is straightforward to extend these results to all indices $i$ by symmetry.

   In particular, near the endpoints both the weights $w_i$ and the distances between the quadrature points $h_i$ are of order $\mathcal{O}(n^{-2})$, while in the middle of the interval their order is $\mathcal{O}(n^{-1})$.
\end{prop}
\begin{proof}
   By \Cref{prop:asymp} and \Cref{lem:sin}
   \[
      w^G_i
      \sim \frac{\pi}{n} \sin \phi^G_i
      \sim \frac{\pi}{n} \phi^G_i
      \sim (i-1/2) \frac{\pi^2}{n^2} \,.
   \]
   The estimate $w^{GL}_i \sim i \frac{\pi^2}{n^2}$ follows that same way.
   Furthermore,
   \[
      h_i^{GL}
      = -\cos{\phi^{GL}_i} +\cos{\phi^{GL}_{i-1}}
      = 2 \sin \frac{\phi^{GL}_i - \phi^{GL}_{i-1}}{2} \sin \frac{\phi^{GL}_i + \phi^{GL}_{i-1}}{2},
   \]
   so
   \[
      h_i^{GL}
      \sim (\phi^{GL}_i - \phi^{GL}_{i-1})  \frac{\phi^{GL}_i + \phi^{GL}_{i-1}}{2}
      \sim \frac{\pi}{n} \frac{2i-1}{2n} \pi
      = (i-1/2) \frac{\pi^2}{n^2},
   \]
   and similarly
   \[
      h_i^{G}
      \sim (\phi^{G}_i - \phi^{G}_{i-1})  \frac{\phi^{G}_i + \phi^{G}_{i-1}}{2}
      \sim \frac{\pi}{n} \frac{4i-4}{4n} \pi
      = (i-1) \frac{\pi^2}{n^2}. \siamqedhere
   \]
\end{proof}

\begin{cor}
   For the appropriately defined indices of Gauss and Gauss--Lobatto points and weights in $[-1,1]$ the following equivalences hold:
   \begin{enumerate}
      \item $w_i^G \sim h_i^{GL} \sim (h_{i+1}^{G}+h_i^{G})/2$
      \item $w_i^{GL} \sim (h_{i+1}^{GL}+h_i^{GL})/2 \sim h_{i+1}^G$ (cf. Lemma 2.1 in \cite{Canuto1994})
      \item $w_i^{GL} \sim (w_{i+1}^{G}+w_i^{G})/2$, $w_{i+1}^{G} \sim (w_{i+1}^{GL}+w_i^{GL})/2$
      \item $h_{i+1}^G \sim (h_{i+1}^{GL}+h_i^{GL})/2$, $h_i^{GL} \sim (h_{i+1}^{G}+h_i^{G})/2$
   \end{enumerate}
\end{cor}

   Numerical results establish that the asymptotic estimates established above are quite sharp in practice.

\subsection{Equivalence of 1D high-order and low-order refined functions}
\label{sec:ho-lor-equivalence}

We next use the quadrature rule estimates to derive norm equivalences between high-order functions, which are polynomials of order $n$ or $n-1$ on $[-1,1]$, e.g.\ with degrees of freedom in the points of an open quadrature rule, and low-order refined functions that are piecewise linear $H^1$ or piecewise-constant $L^2$ functions on the 1D mesh defined by the points of a closed quadrature rule.

First note that since both the $n$-point Gauss and the $(n+1)$-point Gauss--Lobatto rules are exact for polynomials of order $2n-1$, for any polynomial $v$ of order $n-1$ we have
\begin{equation}\label{eq:sec2-eq1}
   \|v\|^2_0 = \int_{-1}^1 v(x)^2 dx
   = \sum_{i=1}^n w_i^G v(x^G_i)^2
   = \sum_{i=0}^n w_i^{GL} v(x^{GL}_i)^2.
\end{equation}

The first equality in \eqref{eq:sec2-eq1} can be written in the form $M_G = D_G$, where $M_G$ is the mass matrix for the nodal basis $\{\psi_i\}$ associated with the points $x_j^G$, i.e. $\psi_i(x_j^G)=\delta_{ij}$, and $D_G$ is the diagonal of $M_G$.
The second equality in \eqref{eq:sec2-eq1} implies that $M_{GL}$ is a rank-one update of its diagonal, $D_{GL}$, and the mass matrix in Gauss--Lobatto points can be preconditioned well by its diagonal \cite{Teukolsky2015}.
This statement holds for many other choices of points, specifically, numerical results show that for $n=1,\ldots 40$:
\begin{itemize}
   \item $\kappa(D_G^{-1}M_G) =1$;
   \item $\kappa(D_C^{-1}M_C) \lesssim 1.12$;
   \item $\kappa(D_{GL}^{-1}M_{GL}) \lesssim 1.5$, remarkably this condition number decreases with $n$ (cf.\ \cite{Teukolsky2015});
   \item $\kappa(D_{\overline{GL}}^{-1}M_{\overline{GL}}) \lesssim 1.27$, where $x_i^{\overline{GL}}$ are the midpoints of the intervals $(x_{i-1}^{GL},x_{i}^{GL})$.
\end{itemize}

Canuto has shown that for polynomials of degree $n$, the $L^2$ norm is well-approximated by the $L^2$ norm of its piecewise linear Gauss--Lobatto interpolant \cite{Canuto1994}, summarized in the following proposition.
\begin{prop} \label{prop:norm-equivalence}
   For any polynomial $v=v_H$ of order $n$ let $v_L$ be the piecewise-linear continuous function which has the same values as $v_H$ in the Gauss--Lobatto points, i.e.
   \[
   v_L(x^{GL}_i)=v_H(x^{GL}_i)\,, \qquad i=0,\ldots,n \,.
   \]
   We have
   \begin{equation} \label{eq:canuto1}
   \|v_H\|_0 \sim \|v_L\|_0
   \qquad\text{and}\qquad
   \|v_H'\|_0 \sim \|v_L'\|_0 \,,
   \end{equation}
   or equivalently
   \begin{equation} \label{eq:canuto2}
   \|v\|_0^2 \sim \sum_{i=0}^n \frac{h^{GL}_i+h^{GL}_{i+1}}{2} \, v(x^{GL}_i)^2
   \qquad\text{and}\qquad
   \|v'\|_0^2 \sim \sum_{i=1}^n \frac{1}{h^{GL}_i} \left(v(x^{GL}_i)-v(x^{GL}_{i-1})\right)^2 .
   \end{equation}
\end{prop}
\begin{proof}
   The estimates \eqref{eq:canuto1} are propositions 2.1 and 2.2 in \cite{Canuto1994} respectively.
   The estimates \eqref{eq:canuto2} follow from the fact that for a linear function $\ell$ on an interval $[a,b]$ we have
   \[
      \int_{a}^b \ell^2 \sim \frac{(b-a)}{2} \left( \ell^2(a) + \ell^2(b) \right)
      \qquad\text{and}\qquad
      \int_{a}^b (\ell')^2 = \frac{1}{(b-a)} \left( \ell(b)-\ell(a) \right)^2 \,.
   \]
   By examining the proofs in \cite{Canuto1994} we notice that the only requirement on the closed set of points is $(h_i+h_{i+1})/2 \sim w_i$.
   The last statement then follows from \Cref{prop:wh-estimates}.
\end{proof}

The next two propositions combine all 1D estimates so far to provide the key
ingredient for the proof of \Cref{lem:R-lower-bound}. Informally it states that
the $L^2$ norm of a 1D polynomial of order $n-1$ is equivalent to the $L^2$ norm
of the piecewise-constant function of its averages on the $n$ intervals defined
by a closed quadrature rule with $n+1$ points.

\begin{prop} \label{prop:lor-equivalence}
   For any polynomial $v$ of order $n-1$ let $v_L$ be the piecewise-constant discontinuous function on the mesh of Gauss--Lobatto points that has the same averages as $v$ on each subinterval $e_i=(x_{i},x_{i+1})$ , i.e.,
   \[
      v_L|_{e_i}= \frac{1}{h_i}\int_{e_i} v \,, \qquad i=1,\ldots,n \,.
   \]
   We have
   \begin{equation} \label{eq:r1}
      \|v\|_0 \sim \|v_L\|_0,
   \end{equation}
   or equivalently
   \begin{equation} \label{eq:r2}
      \|v\|_0^2  \sim \sum_{i=1}^n \frac{1}{h_i} \left(\int_{e_i} v\right)^2.
   \end{equation}
\end{prop}
\begin{proof}
   Let $w$ be a polynomial of order $n$ that satisfies $w'=v$.
   By \eqref{eq:canuto2} applied to $w$ we have
   \[
      \|v\|_0^2
      = \|w'\|_0^2 \sim \sum_{i=1}^n \frac{1}{h_i} \left(w(x_i)-w(x_{i-1})\right)^2
      = \sum_{i=1}^n \frac{1}{h_i} \left(\int_{e_i} v\right)^2.
      \siamqedhere
   \]
\end{proof}

We now want to extend \Cref{prop:lor-equivalence} to more general sets of points.
Notice that the only condition that is required is for $\| w' \|_0 \sim \| w_h' \|_0$, where $w_h$ is the piecewise linear interpolant at the given points.

\begin{prop} \label{prop:general-equivalence}
   Consider any closed set of $(n+1)$ points (i.e.\ containing the interval endpoints) that satisfy condition (2.21) from \cite{Mastroianni2010} (note that this includes both Gauss--Lobatto and Chebyshev Lobatto points).
   Let $w$ be a polynomial of degree $n$, and let $w_h$ be the piecewise linear interpolant of $w$ at these points.
   Then, $\| w' \|_0 \sim \| w_h' \|_0$.
\end{prop}
\begin{proof}
   Of all functions in $H^1[-1,1]$ that interpolate $w$ at the given points, the piecewise linear interpolant $w_h$ has minimum $H^1$ seminorm (cf.\ \cite{Canuto1994}).
   Therefore, $\| w_h' \|_0 \lesssim \| w' \|_0$.
   It remains to show $\| w' \|_0 \lesssim \| w_h' \|_0$.
   By \cite{Bernardi1992}, there exists a polynomial $\pi_n$ of degree $n$ that satisfies the following three properties:
   \begin{align}
      \label{eq:l2-estimate} \| w_h - \pi_n \|_0 &\lesssim n^{-1} \| w_h \|_0, \\
      \label{eq:h1-estimate}\| w_h - \pi_n \|_1 &\lesssim \| w_h \|_1, \\
      \pi_n(\pm 1) &= w_h(\pm 1),
   \end{align}
   where the implicit constants in the inequalities are independent of $n$.
   Then, by property \eqref{eq:h1-estimate},
   \begin{align*}
      \| w' - w_h' \|_0
         \leq \| w_h' - \pi_n' \|_0 + \| w' - \pi_n' \|_0
         \lesssim \| w_h \|_1 + \| w' - \pi_n' \|_0.
   \end{align*}
   Notice that because the interpolation points include $\pm 1$, $w - \pi_n$ is a polynomial of degree $n$ that vanishes at both endpoints.
   Therefore, by the inverse inequality on polynomials (\cite{Bernardi1992}, Lemma 4.4), we have
   $
      \| w' - \pi_n' \|_0 \lesssim n \| w - \pi_n \|_0.
   $
   Then, by property \eqref{eq:l2-estimate},
   \begin{align*}
      \| w - \pi_n \|_0
         \leq \| w - w_h \|_0 + \| \pi_n - w_h \|
         \lesssim \| w - w_h \|_0 + n^{-1} \| w_h \|_0.
   \end{align*}
   Theorem 2.2 from \cite{Mastroianni2010} gives us
   $
      \| w - w_h \|_0 \lesssim n^{-1} \| w_h' \|_0,
   $
   and so, combining the above estimates, we obtain the error estimate
   $
      \| w' - w_h' \|_0 \lesssim \| w_h \|_1
   $
   In particular, we have the stability result
   $
      \| w' \|_0 \lesssim \| w_h \|_1.
   $
   Letting $\overline{w}_h$ denote the average of $w_h$, we apply the Poincar\'e inequality for zero-mean functions (as in \cite{Canuto1994}) to obtain
   $
      \| w' \|_0 = \| (w - \overline{w}_h)' \|_0 \lesssim \| w_h - \overline{w}_h \|_1 \lesssim \| w_h' \|_0.
   $
   We therefore conclude that $\| w' \|_0 \sim \| w_h' \|_0$.
\end{proof}

We are now ready to prove \Cref{lem:R-lower-bound}.
\begin{proof}[Proof of \Cref{lem:R-lower-bound}]
We first consider the case where $V_H$ is the space of polynomials of degree $p$ on $[0,1]$, and $V_L$ is the space of piecewise constant functions defined on the subintervals defined by $n+1$ Gauss--Lobatto points, where $n \geq p + 1$.
Then, defining the operator $R$ by \eqref{eq:R}, we have
\[
   (R v_H, w_L) = (v_H, w_L) \qquad\text{for all $w_L \in V_L$}.
\]
Let $\kappa_L$ denote a given Gauss--Lobatto subinterval.
Choosing $w_L$ to take the value 1 on $\kappa_L$, and 0 elsewhere, we see that $R v_H$ is the piecewise constant function that is equal to the average value of $v_H$ over each Gauss--Lobatto subinterval.
By \Cref{prop:lor-equivalence,prop:general-equivalence}, we see that $\| R v_H \|_0 \sim \| v_H \|_0$, independent of the polynomial degree of the high-order space $p$.
This result trivially extends to the $d$-dimensional cube $[-1,1]^d$ by writing the $d$-dimensional restriction operator $R_d$ as the Kronecker product of the one-dimensional restriction, $R_d = R \otimes \cdots \otimes R$.
Similarly, this estimate can be extended to affine elements with constant Jacobian determinant (cf.\ \Cref{rem:mapping}).

We next consider the case of more general quadrature point sets that satisfy the hypotheses of \Cref{prop:general-equivalence}. Let $V_L$ be the piecewise constant finite element space defined on a low-order refined mesh with subelements defined by such rule. Applying the above result element-by-element, over each element in the high-order space $V_H$, we get
\[
   \| R v_H \|_0 \sim \| v_H \|_0,
\]
and so the estimate $\| R v_H \|_0 \geq \alpha \| v_H \|_0$ holds, with $\alpha = \mathcal{O}(1)$.
\end{proof}

\begin{rem}[Non-affine and curved elements]
   The proof of \Cref{lem:R-lower-bound} holds for affine elements with constant Jacobian determinant.
   In the numerical results in \Cref{sec:numerical}, we consider the more general case of curved elements and mapped geometries given by $\kappa = T([-1,1]^d)$ for diffeomorphism $T : \mathbb{R}^d \to \mathbb{R}^d$.
   The extension of the analysis to this case remains open.
\end{rem}

\begin{rem}[Alternative node sets]
  Empirical results suggest that any set of nodes that is asymptotically distributed according to the \textit{Chebyshev density} $\sim n/(\pi \sqrt{1-x^2})$ will result in accurate transfer operators
   \cite{Krylov1962,Trefethen2008}.
   The numerically computed values of $\alpha$ (i.e.\ the lower bounds of the $R$ operator) for the case of piecewise constant low-order space in one spatial dimension are shown in \Cref{fig:nodes} for a variety of node sets, including uniformly spaced, Gauss--Lobatto, Chebyshev--Lobatto, nodes, as well as the augmented Chebyshev and Gauss--Legendre sets (obtained by taking the union with the interval endpoints $\{-1,1\}$).
   These numerical results suggest that while the lower bound for $R$ degrades severely for uniformly spaced points, it is essentially uniform for other node choices.
\end{rem}

\begin{figure}
   \centering
   \ifsiam
   \includegraphics[width=0.9\linewidth]{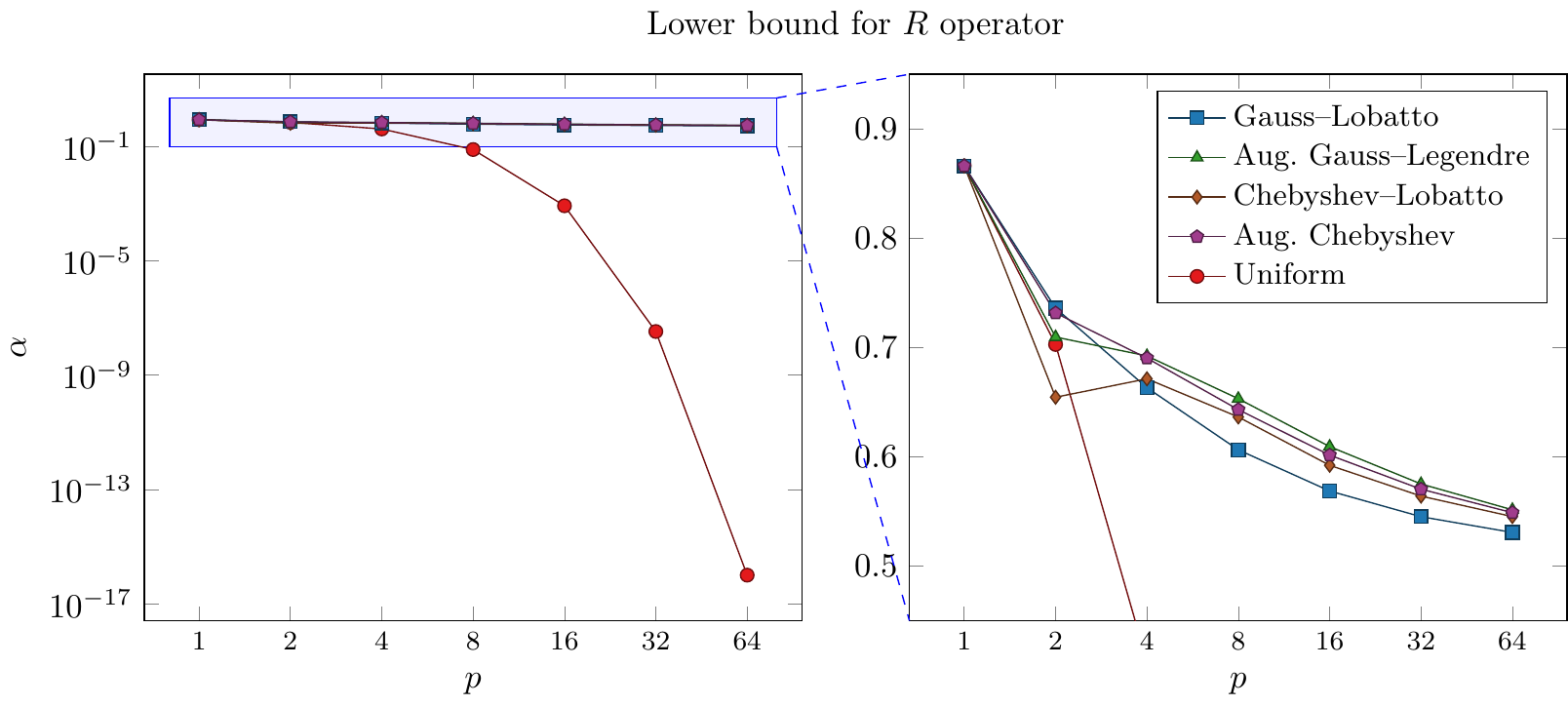}\\
   \else
   \includegraphics[height=2.5in]{fig/quad}\\
   \fi
   \caption{Numerically computed estimates for the lower bound for the $R$ operator for different node choices.
   The right pane in a zoom-in of the plot on the left pane.
   Node sets with Chebyshev distribution remain bounded, while uniformly distributed points exhibit exponential decay of the lower bound.}
   \label{fig:nodes}
\end{figure}

\section{Implementation and numerical results}
\label{sec:numerical}

In this section we discuss the practical implementation of the proposed
transfer operators, particularly with respect to matrix-free efficiency which is
critical for high-order methods. We also present a number of numerical results
confirming the accuracy and conservation analysis in \Cref{sec:accuracy} and
\Cref{sec:mappings}, discuss curved meshes and preconditioning, and
demonstrate the utility of the $R$ and $P$ mappings in the context of adaptive
mesh refinement and conservative multi-discretization coupling.

\subsection{Efficient implementation}
\label{sec:efficient-implementation}

In this section, we describe the efficient implementation of the transfer operators on tensor-product meshes.
In particular, we focus on the high-order matrix-free context, where the computational complexity and storage costs associated with fully assembled matrix-based algorithms are prohibitively expensive.
We will make the assumption that the low-order space $V_L$ is discontinuous, whereas the high-order space $V_H$ can be either continuous or discontinuous.
Let $q$ denote the polynomial degree of the low-order space $V_L$, and let $p$ denote the polynomial degree of $V_H$.
The space $V_L$ is defined on a LOR mesh $\mathcal{T}_L$, which is obtained from the high-order mesh $\mathcal{T}_H$ by subdividing each element into $n^d$ subelements, where $d$ is the spatial dimension, and $n(q+1) \geq p+1$ as required by \eqref{eq:nqp}.

The main tools required for efficient implementation of the transfer operators are \textit{sum factorization}, which allows for the efficient computation of the \textit{action} of the linear operators with optimal memory requirements, and \textit{matrix-free preconditioning}, whereby preconditioners are constructed without access to the entries of the matrix representation of the operator.
We begin with discussion of the restriction operator $R = M_L^{-1} M_{LH}$, since that is also necessary for the computation of the prolongation operator $P$.
The mixed mass matrix $M_{LH}$ can be written as
\[
   M_{LH} = \ME \Lambda,
\]
where $\Lambda$ is the boolean \textit{assembly matrix} that duplicates degrees of freedom shared between elements, and $\ME$ is a block-diagonal matrix whose blocks are the elemental mixed mass matrix.
The blocks of the matrix $\ME$ are of size $(n(q+1))^d \times (p+1)^d$.
In general, each block is dense, and so the memory required to store the assembled mixed mass matrix scales like $(nqp)^d$, and the number of operations required to compute matrix-vector products scales the same.
The number of operations required to form these blocks is $\mathcal{O}\big(n^d q^d p^{2d}\big)$ using naive algorithms, and $\mathcal{O}(n^d q^d p^{d+1})$ using sum factorization techniques \cite{Melenk2001}.

On the other hand, computing the matrix-free action of $\ME$ can be performed in $\mathcal{O}(n^d(p^{d+1} + q^{d+1}))$ operations using sum factorization.
Perhaps more importantly, the memory required to compute the matrix-free action is optimal: assuming that $nq \sim p$, only $\mathcal{O}(p^d)$ memory is required.
As a consequence, the matrix-free algorithm has significantly higher arithmetic intensity than the matrix-based algorithm.
On GPU-based platforms, memory transfer is typically the bottleneck, and the matrix-free algorithms can be expected to outperform algorithms requiring fully assembled matrices \cite{Ljungkvist2017,Franco2020,ceed_bp_paper_2020}.
The appropriate choice of algorithm will depend on both polynomial degrees $p$ and $q$.
In the context of discontinuous Galerkin methods, it has been shown that for moderate orders of $p=3$ or $p=4$, then efficient matrix-free algorithms may significantly outperform the corresponding matrix-based algorithms \cite{Kronbichler2019,Kronbichler2019a}.
However, if the low-order space has polynomial degree $q=0$ or $q=1$, the overhead required for matrix assembly is typically small enough so that matrix-based algorithms remain practical.

The low-order polynomial degree $q$ is typically chosen so that the matrix $M_L$ can be efficiently assembled.
Since the low-order space is discontinuous, the inverse $M_L^{-1}$ can be computed block-by-block using direct methods.
In many practical cases, the low-order space consists of piecewise-constant functions ($q=0$), and so $M_L$ is in fact a diagonal matrix.
In cases where $q$ may be large enough to warrant matrix-free algorithms, an element-by-element preconditioned conjugate gradient algorithm may be used.
In this case, effective diagonal or tensor-product preconditioners ensure uniform convergence of the iterations \cite{Teukolsky2015,Pazner2018e}.

Efficient implementation of the $P$ operator builds on the preceding discussion of the $R$ operator.
We recall that $P = (R^T M_L R)^{-1} R^T M_L$.
The challenging aspect of this operator is performing the action of $A^{-1}$, where $A = R^T M_L R$.
If the high-order space is discontinuous then $A$ is block-diagonal.
In cases where the high-order polynomial degree $p$ is not prohibitively high, this allows for the block-by-block inversion of the operator using direct methods, just as in the case of the DG mass matrix.
For large polynomial degree $p$, it is more efficient to solve the resulting system using a preconditioned conjugate gradient solver.
Furthermore, when the space $V_H$ is continuous, then the system $A$ becomes globally coupled, and block-by-block algorithms are no longer feasible.
In these cases, the matrix-free action of $A$ is performed, as described above.
In \Cref{sec:preconditioning-P}, is shown that any uniform preconditioner for the high-order mass matrix $M_H$ is a uniform preconditioner for $A$.
As a consequence, on tensor-product meshes, the diagonal of the high-order mass matrix is an effective preconditioner for $A$.
Iteration counts using this choice of preconditioner are presented in \Cref{sec:numerical-preconditioning}.

\subsection{Numerical experiments}

The algorithms described in this paper have been implemented in the MFEM finite element library \cite{Anderson2020, mfem-web}, and that implementation was used to perform the numerical experiments presented in this section.
The problems on tensor-product methods make use of the \textit{partial assembly} features of MFEM to implement efficient sum factorized operator action.

\subsubsection{2D test case}

As a first numerical example, we consider a unstructured, straight-sided two-dimensional mesh and high-order $H^1$-conforming finite element space $V_H$ with polynomial degree $p$.
Let $V_L$ denote the piecewise-constant discontinuous space defined on the Gauss--Lobatto refined mesh, where each element is subdivided into $(p+1)^2$ sub-elements.
Consider the function $f$ defined by
\[
   f = \exp(0.1\sin(5.1x - 6.2y) + 0.3\cos(4.3x + 3.4y)).
\]
Let $\Pi_H$ denote $L^2$ projection onto the space $V_H$, and let $R$ and $P$ denote the transfer operators as defined in \Cref{sec:mappings}.
For $g \in  \{ \Pi_H f, R\Pi_H f, PR\Pi_H f \}$, we compute the $L^2$ error $\| f - g \|_0$, and the integral difference $\int_\Omega (f-g) \,dx$.
We consider four uniform refinements of the original mesh, and present the results in \Cref{fig:pw-const-ers}.
As expected by well-known properties of the $L^2$ projection, the $L^2$ error $\| \Pi_H f - f \|_0$ scales like $\mathcal{O}(h^{p+1})$, and the $L^2$ error of the piecewise constant approximation $\| R\Pi_H f - f \|_0$ scales like $\mathcal{O}(h)$.
All of these operations are conservative, and the total integral is preserved up to machine precision for each of the functions.
By \Cref{thm:abstract-properties}, we have $PR=I$, and so the $L^2$ errors for $\Pi_H f$ and $PR\Pi_H f$ are equal.

Additionally, we consider the functions $\Pi_L f$ and $P \Pi_L f$, where $\Pi_L$ denotes $L^2$ projection onto the space $V_L$.
Even though we can only expect $\Pi_L f$ to be first-order accurate, \Cref{thm:P-accuracy} implies that the $L^2$ error of $P \Pi_L f$ will scale like $\mathcal{O}(h^{p+1})$.
This property is verified in \Cref{fig:pw-const-ers}.
Mass conservation is also preserved up to machine accuracy for this test case.

\begin{figure}[ht]
   \centering
   \ifsiam
   \includegraphics[width=\linewidth]{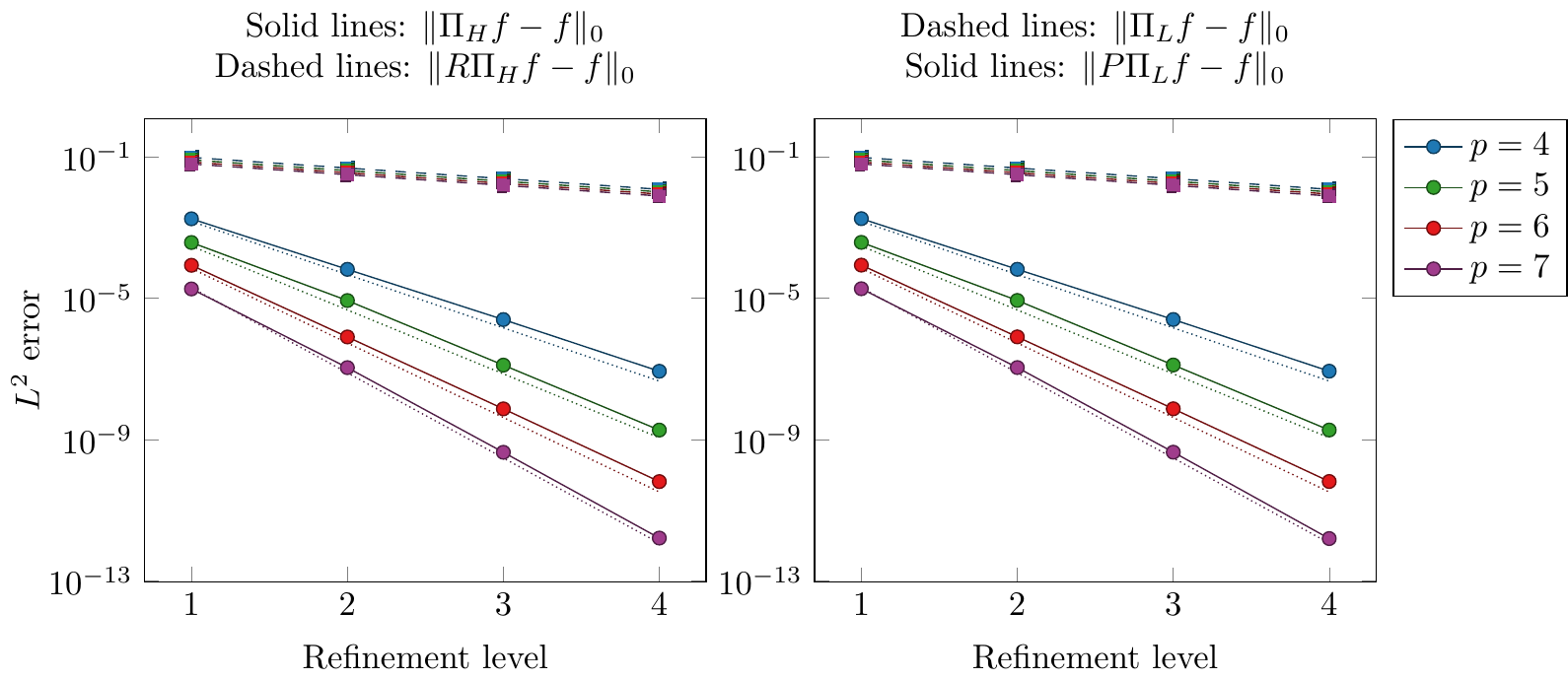}
   \else
   \includegraphics{fig/pwconst_ers}
   \fi
   \caption{
      2D test case.
      $L^2$ errors for transfer operators between high-order finite element space $V_H$ and piecewise-constant low-order space $V_L$.
      Dotted reference lines shown for $\mathcal{O}(h^{p+1})$.
      Dashed lines converge at a rate of $\mathcal{O}(h)$.
   }
   \label{fig:pw-const-ers}
\end{figure}

We also repeat the same 2D test with a piecewise-linear low-order space.
The results are similar to the previous case, but as predicted by the analysis we observe second-order convergence for the low-order functions $R\Pi_H f$ and $\Pi_L f$.

\subsubsection{3D test case}

\begin{figure}[ht]
   \centering
   \ifsiam
   \includegraphics[width=\linewidth]{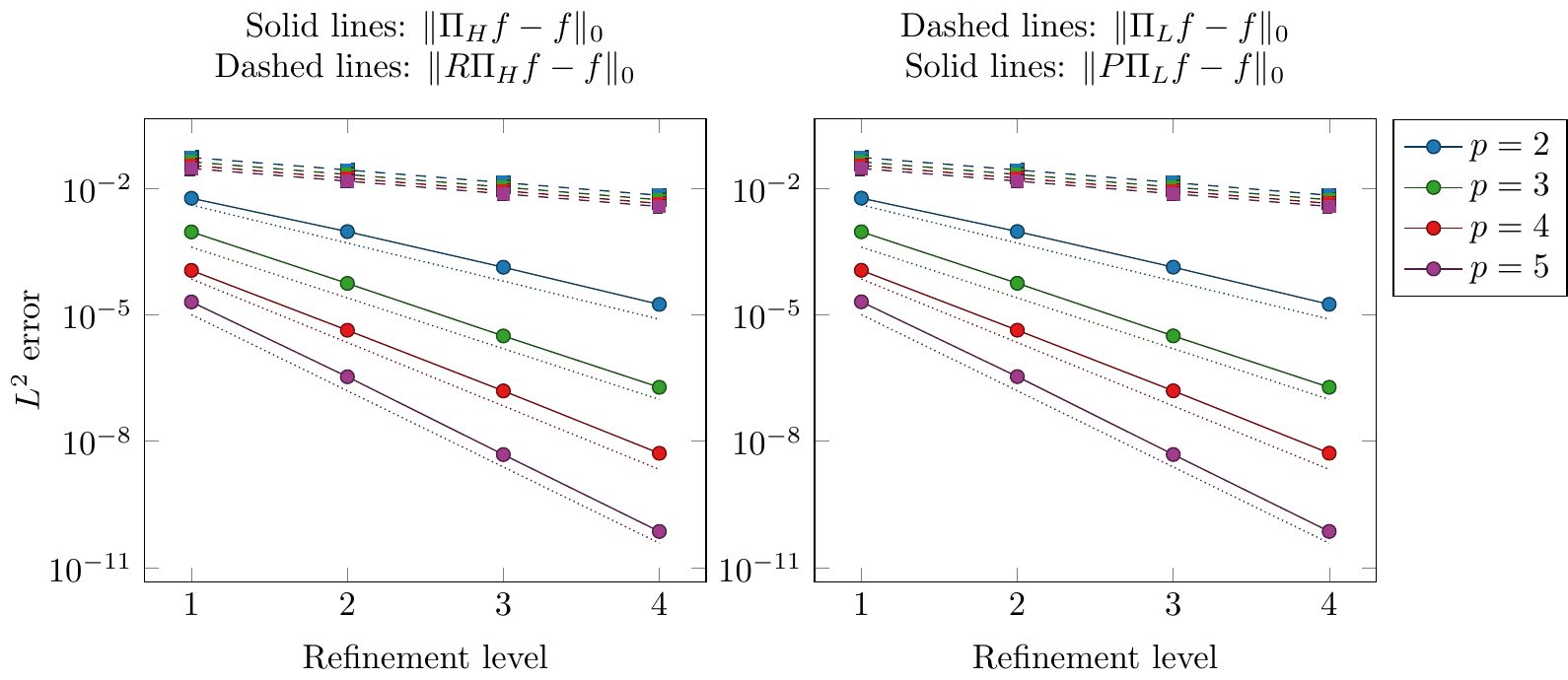}
   \else
   \includegraphics{fig/pwconst_3d_ers}
   \fi
   \caption{
      3D test case.
      $L^2$ errors for transfer operators between high-order finite element space $V_H$ and piecewise-constant low-order space $V_L$.
      Dotted reference lines shown for $\mathcal{O}(h^{p+1})$.
      Dashed lines converge at a rate of $\mathcal{O}(h)$.
   }
   \label{fig:pw-const-3d-ers}
\end{figure}

The analysis and implementation are readily extendable to 3D as shown by the results in \Cref{fig:pw-const-3d-ers}.
The initial mesh for this problem was a $4 \times 4 \times 4$ Cartesian grid, which was then refined uniformly three times to obtain a final mesh of 32{,}768 elements.
Polynomial degrees $p = 2,3,4,5$ were used for the high-order space. The low-order space $V_L$ was taken to be piecewise constant.

\subsection{Curved geometries}
\label{sec:curved}

It is natural for the high-order space $V_H$ to be defined on a high-order (i.e.\ curved) mesh.
While the low-order refined space $V_L$ can in principle be defined on the same curved geometry, it is often advantageous and more natural to define $V_L$ on an associated straight-sided mesh.
For example, if the low-order space $V_L$ is used to transfer solution fields to a low-order discretization that does not support curved meshes, then the mesh must be straight-sided out of practical concerns.
However, the process of converting a curved mesh to straight-sided (e.g.\ by linear interpolation at nodal points) incurs geometric errors.
For example, the total volume and the volumes of individual elements are not guaranteed to be preserved.
In this case, a constant-preserving transfer operator cannot be conservative \cite{Anderson2014}.

To numerically study the performance of the transfer operators defined on curved meshes, we consider a two-dimensional mixed mesh consisting of triangular and quadrilateral elements with mappings defined by degree-3 polynomials.
The mesh is obtained from that shown in \Cref{fig:ho-lor-mesh} by splitting half of the quadrilaterals into triangles and perturbing the mesh nodes.
The high-order space is a degree-5 $H^1$-conforming finite element space defined on the curved mesh, and the low-order refined mesh is a degree-2 $L^2$ finite element space defined on the straight-sided mesh obtained by interpolating the nodal points of the high-order mesh.
(We have $n=4$, $q=2$ and $p=5$, so condition \eqref{eq:nqp} holds.)
The results are presented in \Cref{tab:curved-mesh}.
We observe that although the $L^2$ error of $\Pi_L f$ scales like $\mathcal{O}(h^3)$ (since the low-order space has polynomial degree $q=2$), the restriction $R\Pi_H f$ has $L^2$ error that scales like $\mathcal{O}(h^2)$.
Similarly, the error of $P\Pi_L f$ and the difference in total integrals for both of these quantities scale like $\mathcal{O}(h^2)$.
This is because the total volumes of the high-order and low-order refined meshes differ by $\mathcal{O}(h^2)$.

\begin{table}[ht]
   \centering
   \small
   \caption{Convergence results on two-dimensional mixed mesh with curved elements and piecewise-quadratic low-order space ($n=4$, $q=2$ and $p=5$).}
   \label{tab:curved-mesh}
   \begin{tabular}{r|ccc|cccc}
      \toprule
      & \multicolumn{3}{c}{$\Pi_H f$} & \multicolumn{4}{c}{$R\Pi_H f$}\\
      Ref. & $L^2$ error & Rate & Integral & $L^2 error$ & Rate & Integral & Rate\\
      \midrule
      0 & $4.80\times10^{-4}$ & --- & $1.07\times10^{-14}$ &
      $3.69\times10^{-3}$ & --- & $8.72\times10^{-4}$ & ---\\
      1 & $1.12\times10^{-5}$ & 5.43 & $1.15\times10^{-14}$ &
      $8.48\times10^{-4}$ & 2.12 & $1.82\times10^{-4}$ & 2.26\\
      2 & $1.91\times10^{-7}$ & 5.87 & $4.44\times10^{-15}$ &
      $2.11\times10^{-4}$ & 2.01 & $4.42\times10^{-5}$ & 2.04\\
      3 & $2.84\times10^{-9}$ & 6.07 & $1.95\times10^{-14}$ &
      $5.27\times10^{-5}$ & 2.00 & $1.10\times10^{-5}$ & 2.01\\
      \bottomrule
   \end{tabular}

   \vspace{\floatsep}

   \begin{tabular}{r|ccc|cccc}
      \toprule
      & \multicolumn{3}{c}{$\Pi_L f$} & \multicolumn{4}{c}{$P\Pi_L f$}\\
      Ref. & $L^2$ error & Rate & Integral & $L^2$ error & Rate & Integral & Rate\\
      \midrule
      0 & $1.05\times10^{-3}$ & --- & $1.87\times10^{-14}$ &
      $3.23\times10^{-3}$ & --- & $8.80\times10^{-4}$ & ---\\
      1 & $1.41\times10^{-4}$ & 2.90 & $2.66\times10^{-14}$ &
      $7.51\times10^{-4}$ & 2.10 & $1.83\times10^{-4}$ & 2.27\\
      2 & $1.77\times10^{-5}$ & 2.99 & $5.95\times10^{-14}$ &
      $1.89\times10^{-4}$ & 1.99 & $4.42\times10^{-5}$ & 2.05\\
      3 & $2.22\times10^{-6}$ & 3.00 & $4.17\times10^{-14}$ &
      $4.72\times10^{-5}$ & 2.00 & $1.10\times10^{-5}$ & 2.01\\
      \bottomrule
   \end{tabular}
\end{table}

\subsection{Preconditioning}
\label{sec:numerical-preconditioning}

In this example, we consider the diagonal preconditioning of $P$ defined in \Cref{sec:preconditioning-P}.
Recall that $P = (R^T M_L R)^{-1} R^T M_L$.
\Cref{prop:preconditioning-P} established that the condition number of $D^{-1} (R^T M_L R)$, where $D$ is the diagonal of the high-order mass matrix, is independent of the mesh size and polynomial degree.
We consider the two-dimensional straight-sided mesh shown in \Cref{fig:ho-lor-mesh}, with four levels of uniform refinement.
The high-order space is a $H^1$-conforming finite element space with polynomial degree $p$, and $V_L$ is a $L^2$ low-order refined finite element space with polynomial degree $q$.
We record the number of conjugate gradient iterations required to converge to a relative tolerance of $10^{-12}$ in \Cref{tab:preconditioning}.
We observe that the iteration counts remain bounded both with increasing refinements and increasing high-order polynomial degree $p$.
Additionally, the iteration counts with low-order polynomial degree $q=1$ are uniformly lower than the corresponding iteration counts with $q=0$.

\begin{table}[ht]
   \centering
   \small
   \caption{Conjugate gradient iterations required to solve $R^T M_L R$ with diagonal preconditioning for a relative tolerance of $10^{-12}$.}
   \label{tab:preconditioning}
   \begin{tabular}{c|ccccc|ccccc}
      \toprule
      & \multicolumn{5}{c}{$q=0$} & \multicolumn{5}{c}{$q=1$} \\
      Ref. & $p=1$ & $p=2$ & $p=3$ & $p=4$ & $p=5$ & $p=1$ & $p=2$ & $p=3$ & $p=4$ & $p=5$ \\
      \midrule
      0 & 23 & 37 & 38 & 39 & 35 & 21 & 22 & 22 & 21 & 18 \\
      1 & 39 & 38 & 36 & 32 & 31 & 32 & 23 & 20 & 17 & 14 \\
      2 & 41 & 35 & 30 & 29 & 28 & 31 & 21 & 17 & 14 & 11 \\
      3 & 37 & 32 & 27 & 27 & 25 & 29 & 19 & 14 & 10 & 9  \\
      \bottomrule
   \end{tabular}
\end{table}

\subsection{AMR coarsening}
\label{sec:amr}

Consider a conforming mesh, obtained from a coarse mesh through a series of uniform refinements.
This mesh is then further refined through a series of non-conforming (potentially anisotropic) refinements, resulting in a non-matching mesh with hanging nodes \cite{Cerveny2019}.
These refinements could be driven through an adaptive process; in this example, the refinements are performed randomly.
Let $V_C$ denote a degree-$p$ finite element space on the conforming mesh, and $V_{\NC}$ denote the degree-$p$ space on the nonconforming mesh.
Since $V_C \subseteq V_{\NC}$, the natural injection $R : V_C \hookrightarrow V_{\NC}$ satisfies the properties of the $R$ operator as defined in \Cref{thm:abstract-properties}.
Defining the $P : V_{\NC} \to V_C$ operator as in \Cref{thm:abstract-properties} gives a method for \textit{coarsening} a field defined on $V_{\NC}$.
Let $M_{\NC}$ denote the mass matrix defined on the space $V_{\NC}$.
Then, the $P$ operator takes the form $P = (R^T M_{\NC} R)^{-1} R^T M_{\NC}$.
As in the cases of the low-order refined transfer operators, computing the action of $A^{-1}$, where $A = R^T M_{\NC} R$, generally requires a globally coupled solve.
The arguments of \Cref{sec:preconditioning-P} apply also to this operator: $A$ is symmetric and positive-definite, and $A$ is spectrally equivalent to $M_C$, the mass matrix defined on the conforming space.
We therefore use the diagonal of the mass matrix defined on the conforming mesh as a preconditioner for the operator $R^T M_{\NC} R$.

\begin{rem}[Matrix-free implementation of AMR coarsening]
   As discussed in \Cref{sec:efficient-implementation}, the matrix-free action of $M_{\NC}$ can be performed efficiently using sum factorization techniques.
   Similarly, the diagonal of $M_C$, required for preconditioning, can be obtained using matrix-free algorithms.
   Therefore, given an efficient matrix-free representation of the injection $R : V_C \hookrightarrow V_{\NC}$, the coarsening operator $P$ also has an efficient matrix-free implementation.
\end{rem}

To numerically study the behavior of this coarsening operator, we consider a fixed two dimensional mesh, and perform $\ell$ uniform refinements.
Subsequently, a sequence of random refinements is performed to obtain the nonconforming mesh.
Let $V_{C}$ and $V_{\NC}$ be degree-5 $H^1$-conforming spaces defined on the conforming and nonconforming meshes, respectively.
The function $u_h \in V_{\NC}$ is obtained by interpolating a given function $f$ at nodal points.
Then, a coarsened function $P u_h \in V_C$ is obtained by applying the coarsening operator $P$.
An example of this transfer process is illustrated in \Cref{fig:amr}.
In \Cref{tab:amr}, we present convergence results for the transferred solution $P u_h$.
We note that $\| P u_h - f \|_0 = \mathcal{O}(h^{p+1})$.
Additionally, we compute the conservation error by comparing the integrals of the solutions $u_h$ and $P u_h$.
Verifying the conservation properties of \Cref{thm:abstract-properties}, we see that the transfer operator is conservative up to machine precision.

\begin{figure}
   \centering
   \ifsiam\setlength{\fw}{1.25in}\else\setlength{\fw}{1.5in}\fi
   \begin{minipage}{\fw}
      \centering
      \includegraphics[width=\linewidth]{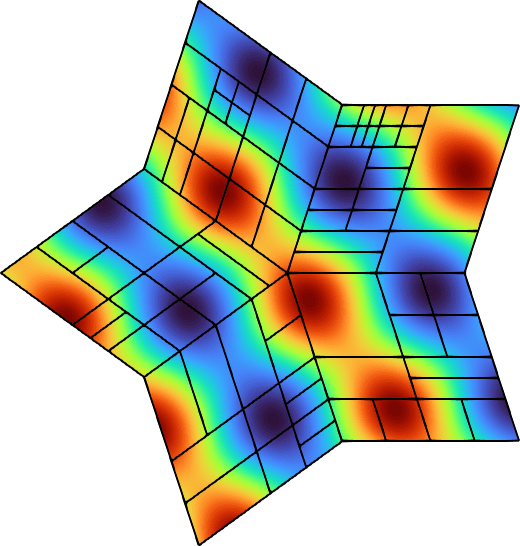}\\[12pt]
      $u_h \in V_{\NC}$
   \end{minipage}
   \hspace{0.25in}
   {\Large $\xrightarrow{\mathmakebox[2em]{P}}$}
   \hspace{0.25in}
   \begin{minipage}{\fw}
      \centering
      \includegraphics[width=\linewidth]{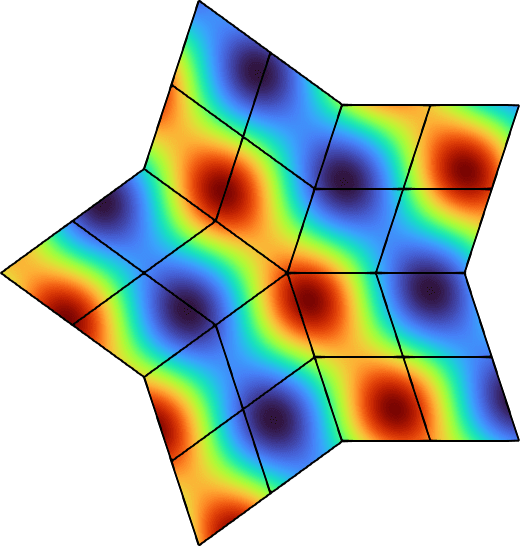}\\[12pt]
      $P u_h \in V_C$
   \end{minipage}
   \caption{Illustration of coarsening operator $P$, mapping from the finite element space $V_{\NC}$ defined on the nonconforming mesh (left) to the coarser space $V_C$ defined on the conforming mesh (right).}
   \label{fig:amr}
\end{figure}

\begin{table}
   \centering
   \small
   \caption{Convergence results for AMR coarsening.}
   \label{tab:amr}
   \begin{tabular}{r|cc|cc|c}
      \toprule
      Ref. & $\|u_{h} - f\|_0$ & Rate & $\|P u_{h} - f\|_0$ & Rate & $\left| \int_\Omega (u_{h} - Pu_{h}) \, dx \right|$ \\
      \midrule
      0 & $3.02\times10^{-4}$ & --- & $4.32\times10^{-4}$ & --- & $7.99\times10^{-15}$ \\
      1 & $5.96\times10^{-6}$ & 5.66 & $9.46\times10^{-6}$ & 5.51 & $8.88\times10^{-15}$ \\
      2 & $9.10\times10^{-8}$ & 6.03 & $1.43\times10^{-7}$ & 6.05 & $6.48\times10^{-14}$ \\
      3 & $1.60\times10^{-9}$ & 5.83 & $2.18\times10^{-9}$ & 6.03 & $1.10\times10^{-13}$ \\
      4 & $2.44\times10^{-11}$ & 6.04 & $3.31\times10^{-11}$ & 6.04 & $2.66\times10^{-14}$ \\
      \bottomrule
   \end{tabular}
\end{table}

\subsection{Conservative multi-discretization coupling}

To demonstrate the utility of these transfer operators for multiphysics or multi-discretization applications, we consider the coupling of a high-order finite element method to a high-order structured finite volume method.
The space $V_H$ is chosen to be a degree-$p$ piecewise polynomial finite element space defined on a two-dimensional Cartesian grid denoted $\mathcal{T}_H$ of the spatial domain $\Omega$.
In principle, the space $V_H$ can be chosen to be either a continuous Galerkin or discontinuous Galerkin space; in this example, we choose $V_H$ to be a continuous space.
The low-order space $V_L$ is a piecewise constant (i.e. finite volume) space defined on a mesh $\mathcal{T}_L$, which is obtained from $\mathcal{T}_H$ through uniform refinements.
Each element of $\mathcal{T}_H$ is subdivided into at least $(p+1)^2$ sub-elements.

Given $u_H^0 \in V_H$, which could be obtained, for example, through the solution of a high-order finite element problem, we compute $u_L^0 = R u_H^0$.
This piecewise-constant field is used as the initial condition for a finite volume discretization of the scalar advection equation $u_t + \nabla\cdot(\bm\beta u) = 0$.
The finite volume discretization evolves the cell averages $\overline{u}_i$ by integrating reconstructed polynomials on cell faces using an upwind numerical flux.
The initial condition is integrated in time using the standard fourth-order Runge--Kutta method to obtain the solution $u_L^N$.
Using degree-$q$ polynomial reconstructions in the finite volume discretization, the spatial error scales as $\mathcal{O}^(h^{q+1})$, where $h$ is the element size of the mesh $\mathcal{T}_L$.
Additionally, since we use a conservative finite volume method, the total mass is conserved, i.e.
$\int_\Omega u_L^{0} \, dx = \int_\Omega u_L^{N} \, dx.$
The piecewise-constant field $u_L^N$ is transferred to the high-order finite element space using the prolongation operator $P$, i.e.\ $u_H^N = P u_L^N$.
The conservation properties of the transfer operators (\Cref{thm:abstract-properties}) and the accuracy of the prolongation operator (\Cref{thm:P-accuracy}) guarantee that the solution $u_H^N$ in the high-order finite element space will have accuracy $\mathcal{O}(h^{\min\{p+1,q+1\}})$, and that the total mass will be conserved, $\int_\Omega u_H^{0} \, dx = \int_\Omega u_H^{N} \, dx.$

We numerically verify these properties on the unit square $\Omega = [0,1]^2$, with the rotational velocity field $\bm\beta = (2y-1, 1-2x)^T$ and periodic boundary conditions.
The initial condition is taken to be the sum of two Gaussian perturbations,
\[
   u^0 =
   \exp\left(-200\left(
      \left(x-\tfrac{1}{4}\right)^2 + \left(y-\tfrac{1}{2}\right)^2
   \right)\right)
   +
   \exp\left(-200\left(
      \left(x+\tfrac{1}{4}\right)^2 + \left(y-\tfrac{1}{2}\right)^2
   \right)\right).
\]
The high-order initial condition $u_H^0$ is obtained by interpolating $u^0$ at the Gauss--Lobatto nodes of the high-order space $V_H$, such that $\| u_H^0 - u^0 \|_0 = \mathcal{O}(h^{p+1})$.
We use polynomial degree $p=2$, subdivide each mesh element into $4^2$ subelements (so that $V_L$ has strictly more degrees of freedom than $V_H$), and use a fourth-order finite volume method.
The equations are integrated in time for one quarter revolution, until $t=\pi/4$.
Snapshots of the initial and final solutions are shown in \Cref{fig:fv} and convergence results are displayed in \Cref{tab:fv}.
We begin with a $10 \times 10$ Cartesian grid, and refine uniformly four times to compute the observed rates of convergence.
The $L^2$ error of the piecewise-constant solution scales as $\mathcal{O}(h)$.
However, since a fourth-order finite volume reconstruction is used, the $L^2$-norm difference between the finite volume solution $u_L^N$ and the $L^2$ projection of the exact solution $\Pi_L u|_{t=\pi/4}$ scales as $\mathcal{O}\left(h^{\max\{p+1,q+1\}}\right) = \mathcal{O}(h^3)$ (that is to say, the finite volume cell averages $\overline{u}_i$ provide high-order approximations to the cell averages of the true solution).
This high-order of accuracy is preserved when transferring the finite volume solution back to the high-order finite element space; we observe $\mathcal{O}(h^{p+1})$ convergence in $L^2$ norm of $P u_L^N$.
Additionally, we verify that the conservation error remains at the level of machine precision for all test cases performed.

\begin{figure}[ht]
   \centering
   \def\w{5}
   \begin{tikzpicture}[scale=0.8]
      \node[anchor=south] at (0,0) {\includegraphics[width=\w cm]{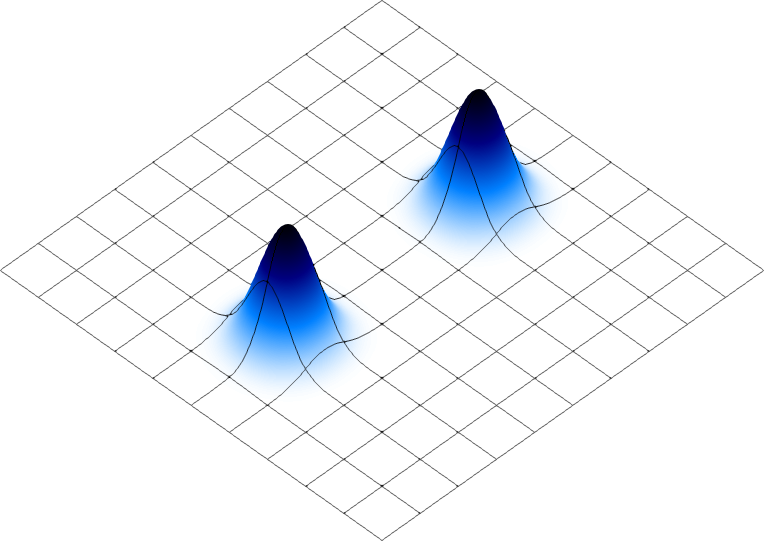}};
      \node at (0,-0.5) {$u_H^0 \in V_H$};

      \draw [thick,->] (0.5*\w + 0.5 ,1.5) -- ++(1.25,0) node[midway,above] {\large $R$};

      \node[anchor=south] at (1.5*\w,0) {\includegraphics[width=\w cm]{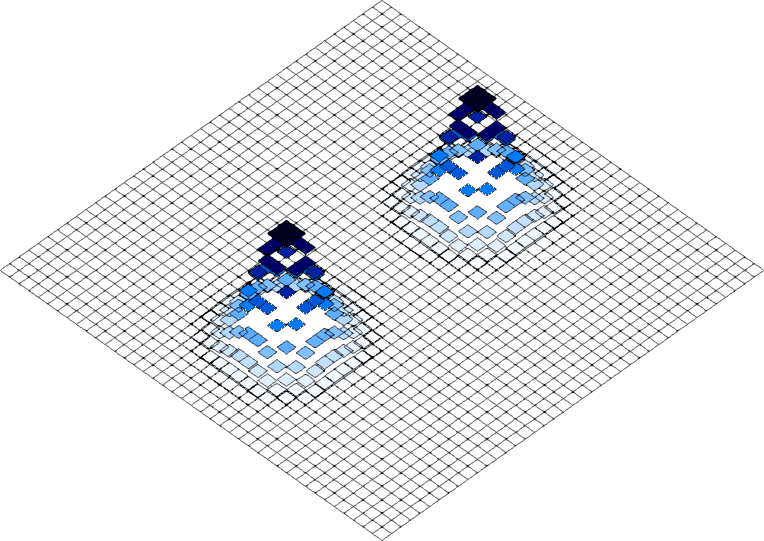}};
      \node at (1.5*\w,-0.5) {$u_L^0 = R u_H^0 \in V_L$};

      \draw [thick,->] (1.5*\w, -1.25) -- ++(0,-0.25*\w)
      node[draw,thin,midway,right,xshift=12pt,align=center,rounded corners=5,inner sep=5]
      {Finite volume\\time evolution};

      \begin{scope}[shift={(0,-1.5*\w)}]
         \node[anchor=south] at (0,0) {\includegraphics[width=\w cm]{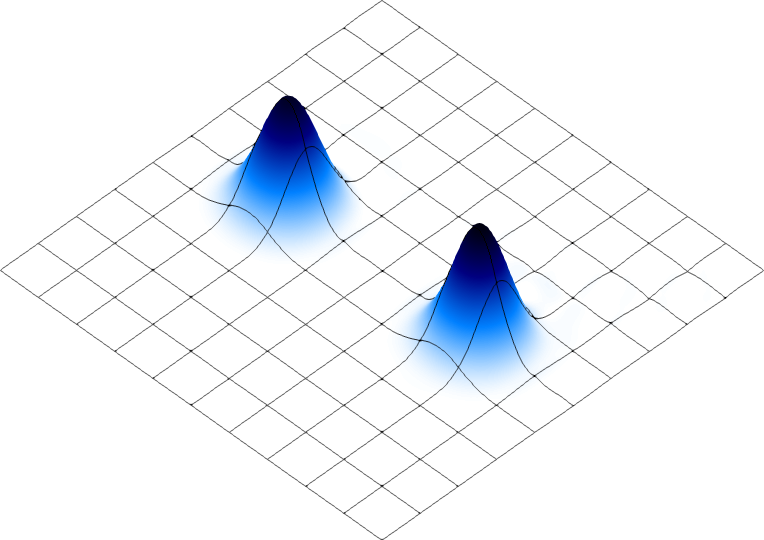}};
         \node at (0,-0.5) {$u_H^N = P u_L^N \in V_H$};

         \draw [thick,<-] (0.5*\w + 0.5 ,1.5) -- ++(1.25,0) node[midway,above] {\large $P$};

         \node[anchor=south] at (1.5*\w,0) {\includegraphics[width=\w cm]{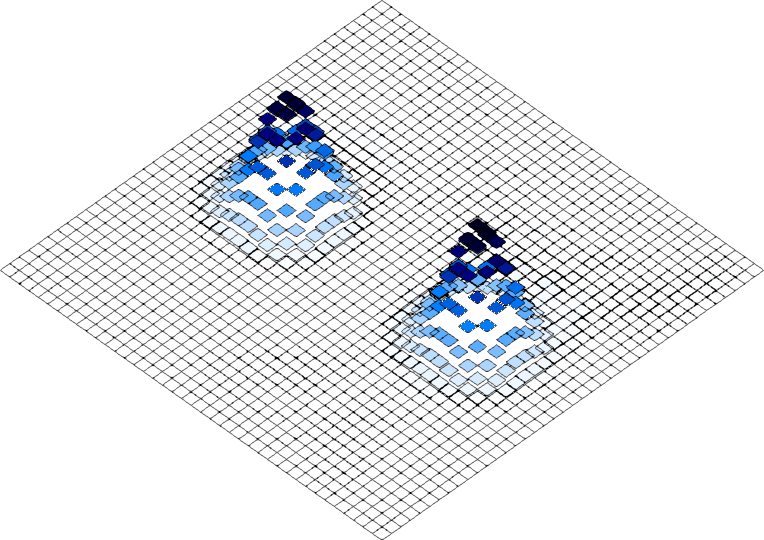}};
         \node at (1.5*\w,-0.5) {$u_L^N \in V_L$};
      \end{scope}
   \end{tikzpicture}

   \caption{
      High-order and conservative solution transfer between a high-order finite element representation and a finite volume representation.
      Starting with restricted high-order initial condition, $u_L^0 = R u_H^0$, a high-order finite volume reconstruction is used to integrate the advection equation in time $u_L^0 \mapsto u_L^N$, obtaining a solution which is transferred to the finite element space using the prolongation operator $u_H^N = P u_L^N$.
   }
   \label{fig:fv}
\end{figure}
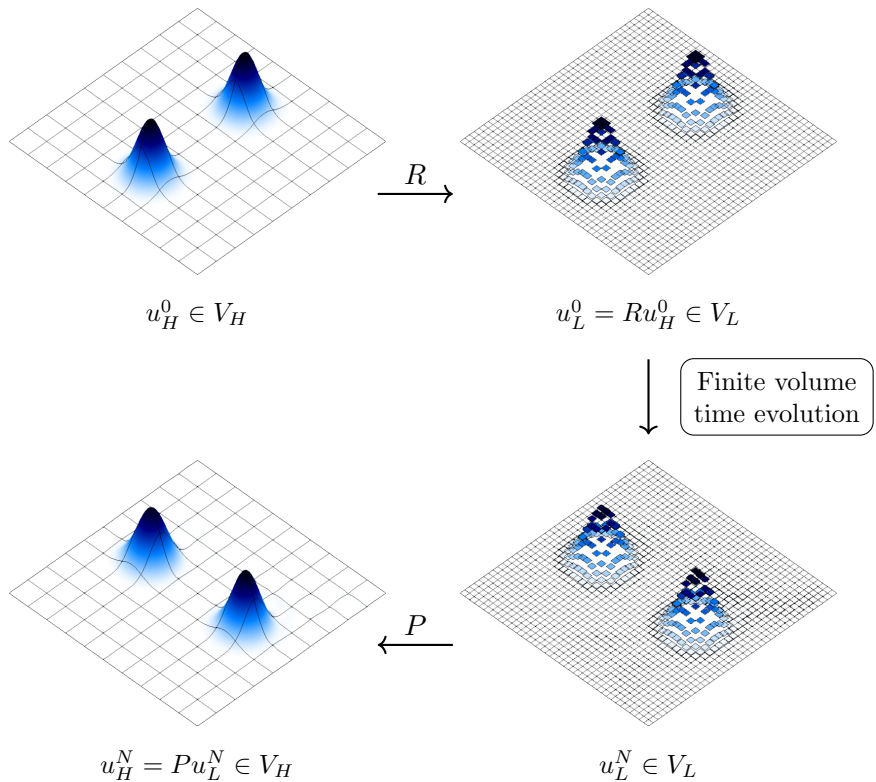

\begin{table}
   \centering
   \footnotesize
   \caption{Convergence results for multi-discretization coupling.}
   \label{tab:fv}
   \begin{tabular}{r|cc|cc|cc|c}
      \toprule
      $n_x$ & $\|u_L^N - u\|_0$ & Rate & $\|u_L^N - \Pi_L u\|_0$ & Rate & $\| P u_L^N - u \|_0$ & Rate & $\left| \int_\Omega (u_H^N - u_H^0) \, dx \right|$ \\
      \midrule
      10 & $2.12\times10^{-2}$ & --- & $1.14\times10^{-1}$ & --- & $1.26\times10^{-2}$ & --- & $6.94\times10^{-18}$ \\
      20 & $9.12\times10^{-3}$ & 1.22 & $1.60\times10^{-2}$ & 2.84 & $1.58\times10^{-3}$ & 3.00 & $8.19\times10^{-16}$ \\
      40 & $4.52\times10^{-3}$ & 1.01 & $1.47\times10^{-3}$ & 3.44 & $1.90\times10^{-4}$ & 3.05 & $3.68\times10^{-16}$ \\
      80 & $2.26\times10^{-3}$ & 1.00 & $1.95\times10^{-4}$ & 2.92 & $2.44\times10^{-5}$ & 2.96 & $3.33\times10^{-16}$ \\
      160 & $1.13\times10^{-3}$ & 1.00 & $2.40\times10^{-5}$ & 3.02 & $3.07\times10^{-6}$ & 2.99 & $2.35\times10^{-15}$ \\
      \bottomrule
   \end{tabular}
\end{table}

\section{Conclusions}
\label{sec:conclusions}

In this paper, we introduced solution transfer operators between high-order and low-order finite element spaces.
The operators are defined in a general, abstract context, but particular attention is paid to the case of \textit{low-order refined} spaces, whereby the low-order finite element space is obtained by refining the coarse elements of the original high-order mesh.
The transfer operators are shown to be conservative, constant preserving, and accurate.
In particular, we show that when Gauss--Lobatto nodes are used to define the low-order refined mesh, the accuracy of the prolongation operator does not degrade as the polynomial degree is increased.
Efficient implementation techniques, including sum factorization and matrix-free preconditioning are discussed.
The theoretical properties, including accuracy and conservation are illustrated with a number of numerical examples.

\section*{Acknowledgements}

We gratefully acknowledge the valuable contributions of V.\ Dobrev to both the initial conceptualization and implementation in MFEM of the transfer operators described in \Cref{sec:mappings}.

This work was performed under the auspices of the U.S.\ Department of Energy by Lawrence Livermore National Laboratory under Contract DE-AC52-07NA27344 (LLNL-JRNL-819814).
This document was prepared as an account of work sponsored by an agency of the United States government.
Neither the United States government nor Lawrence Livermore National Security, LLC, nor any of their employees makes any warranty, expressed or implied, or assumes any legal liability or responsibility for the accuracy, completeness, or usefulness of any information, apparatus, product, or process disclosed, or represents that its use would not infringe privately owned rights.
Reference herein to any specific commercial product, process, or service by trade name, trademark, manufacturer, or otherwise does not necessarily constitute or imply its endorsement, recommendation, or favoring by the United States government or Lawrence Livermore National Security, LLC.
The views and opinions of authors expressed herein do not necessarily state or reflect those of the United States government or Lawrence Livermore National Security, LLC, and shall not be used for advertising or product endorsement purposes.

\ifsiam
\small
\bibliographystyle{abbrv}
\else
\bibliographystyle{siamplain}
\fi
\bibliography{refs}

\end{document}